\documentclass[12pt,a4paper,leqno,oneside]{amsart}
\usepackage[utf8]{inputenc}
\usepackage[a4paper,hmargin={2.3cm,2.3cm},vmargin={3cm,3cm}]{geometry}
\usepackage{amsmath}
\usepackage{amssymb}
\usepackage{amsthm}
\usepackage{graphicx}
\usepackage{hyperref}

\newtheorem{theorem}{Theorem}[section]
\newtheorem{corollary}[theorem]{Corollary}

\newtheorem{proposition}[theorem]{Proposition}
\newtheorem{lemma}[theorem]{Lemma}

\numberwithin{equation}{section}

\theoremstyle{remark}
\newtheorem{remark}[theorem]{Remark}

\def\Ran{\mathop{\rm Ran}\nolimits}
\def\capacity{\mathop{\rm cap}\nolimits}
\def\diam{\mathop{\rm diam}\nolimits}
\def\dist{\mathop{\rm dist}\nolimits}

\def\Max{{\mathrm{max}}}

\def\good{{\mathrm{good}}}

\def\d{\mathrm{d}}
\def\barP{\bar{\mathbb P}}
\def\barE{\bar{\mathbb E}}
\def\bP{{\mathbf P}}

\def\bbP{{\bar{\mathbf P}}}

\def\bbV{{\mathbf V}}

\def\<{\langle}
\def\>{\rangle}

\begin{document}

% TITLE %<<<
\title[Critical window for the vacant set]{Critical window for the vacant
  set left by random walk on random regular graphs}

  \author[J. Černý]{Jiří Černý$^1$}
  %\address
  \thanks{$^1$ Department of Mathematics,
    ETH Zurich, Raemistrasse 101, 8092 Zurich, Switzerland}
  %\email{cerny@math.ethz.ch}

  \author[A. Teixeira]{Augusto Teixeira$^2$}
  %\address
  \thanks{$^2$ Department of Mathematics and their Applications,
    Ecole Normale Supérieure,
    45 rue d'Ulm, F-75230 Paris, France}
  %\email{augusto.teixeira@ens.fr}

  \thanks{This research was partially supported by the AXA Research Fund
    Fellowship. A.T.~thanks the
    Forschungsinstitut für Mathematik (FIM) of ETH Zurich, where a part
    of this work has been completed, for its hospitality.}

  \date{10 January 2011}

\begin{abstract}
  We consider the simple random walk on a random $d$-regular
  graph with $n$ vertices, and investigate
  percolative properties of the set of vertices not visited by the walk
  until time $un$, where $u>0$ is a fixed positive parameter. It was shown
  in \cite{CTW11} that this so-called vacant set exhibits a phase
  transition at $u=u_\star$: there is a giant component if $u<u_\star$ and
  only small components when $u>u_\star$. In this paper we show the
  existence of a critical window of size $n^{-1/3}$ around $u_\star$.
  In this window the size of the largest cluster is of order~$n^{2/3}$.
\end{abstract}

\maketitle%>>>

%%%% START OF 01-intro.tex %%%%<<<
\section{Introduction}

The study of percolative properties of the vacant set left by a random
walk on  finite graphs was initiated  by Benjamini and Sznitman
\cite{BS08} for the case of random walk on a high-dimensional discrete
torus $(\mathbb{Z}/N \mathbb{Z})^d$.  In \cite{BS08} it is proved that if
the random walk runs up to time $uN^d$, where $u$ is a small constant,
the vacant set has a giant component with volume of order $N^d$,
asymptotically as $N$ grows. On the other hand, if instead we consider a
large constant $u$, all components of the vacant set have a volume of
order at most $\log^\lambda N$ (for some $\lambda > 0$) as was proved in
\cite{TW10}. This shows the existence of two distinct phases for the
connectivity of the vacant set on the torus as $u$ varies. However, the
above mentioned works leave several open questions, such as whether the
transition between these two phases happens abruptly at a given critical
threshold and, if this is the case, how does the vacant set behave at the
critical point?

Such problems are much better understood when instead of the torus one
considers random $d$-regular graphs or $d$-regular large-girth expanders
on $n$ vertices. In this case, when random walks runs up to time $un$, it
is known that the vacant set exhibits a sharp phase transition
\cite{CTW11}: the size of the largest connected component of the vacant
set drops abruptly from order $n$ to order $\log n$ at a computable
critical value $u_\star$. In this paper we explore more closely this
phase transition, in particular we prove that the size of the largest
component of the vacant set exhibits a double-jump, similar to that
observed in Erdős-Rényi random graphs.

Let us now give a precise definition of the model. Let $\mathcal G_{n,d}$
be the set of all non-oriented $d$-regular simple graphs with $n$
vertices (here and later we tacitly assume that $n d$ is even).  Let
$\mathbb P_{n,d}$ be the uniform probability distribution on
$\mathcal G_{n,d}$. For any graph $G=(V,\mathcal E)$ let $P^G$ denote the
canonical law on the Skorokhod space $D([0,\infty),V)$ of a
continuous-time simple random walk on $G$ started from the uniform
distribution. We use $(X_t)_{t\ge 0}$ to denote the canonical coordinate
process. For a fixed parameter $u\ge 0$, we define the vacant set as the
set of all vertices not visited by the random walk before the time $u|V|$,
\begin{equation}
  \label{e:Vu}
  \mathcal V^u=\mathcal V^u_G=V\setminus \{X_t:0\le t\le u|V|\}.
\end{equation}
We use $\mathcal C^u_\Max$ to denote the maximal connected component of
the subgraph of $G$ induced by $\mathcal V^u$. The vacant set and in
particular its maximal connected component are the main objects of
investigation in this paper.

As proved in \cite{CTW11}, the phase transition in the connectivity of
the vacant set occurs at the value $u_\star$ given by
\begin{equation}
  \label{e:ustar}
  u_\star= \frac{d(d-1) \ln(d-1)}{(d-2)^2},
\end{equation}
and can be described as follows: Let $G_n$ be a graph distributed
according to $\mathbb P_{n,d}$. Then with $\mathbb P_{n,d}$-probability
tending to one as $n\to\infty$:
\begin{description}
  \item[Super-critical phase] For any $u<u_\star$ and $\sigma >0$ there exist
  $\rho$ and $c$ depending on $u$, $\sigma $, and $d$, such that
  \begin{equation}
    \label{e:supercritical}
    P^{G_n}[|\mathcal C^u_\Max|\ge \rho  n]\ge 1- cn^{-\sigma }.
  \end{equation}
  \item[Sub-critical phase] For any $u>u_\star$ and $\sigma >0$ there
  exist $K$ and $c$ depending on $u$, $\sigma $, and $d$, such that
  \begin{equation}
    \label{e:subcritical}
    P^{G_n}[|\mathcal C^u_\Max|\le K\log n]\ge 1- c n^{-\sigma }.
  \end{equation}
\end{description}

In this paper we study the behaviour of the vacant set in the vicinity of
the critical point. The main results of this paper are the following two
theorems. In their statement we use $\bP_{n,d}$ to denote the
\emph{annealed} measure
\begin{equation}
  \label{e:annealed}
  \bP_{n,d} (\cdot) = \int P^G(\cdot) \mathbb P_{n,d}(\d G).
\end{equation}
We say that an event $A$ occurs $\bP_{n,d}$-asymptotically almost surely
(or simply $\bP_{n,d}$-a.a.s.)~if $\lim_{n\to\infty}\mathbb \bP_{n,d}(A)=1$.

\begin{theorem}[Critical window]
  \label{t:criticalwindow}
  Let $(u_n)_{n\ge 1}$ be a sequence satisfying
  \begin{equation}
    |n^{1/3}(u_n-u_\star)|\le\lambda<\infty \quad
    \text{for all $n$ large enough.}
  \end{equation}
  Then for every $\varepsilon >0$ there
  exists $A=A(\varepsilon ,d,\lambda )$ such that for all $n$ large enough
  \begin{equation}
    \bP_{n,d}[A^{-1}n^{2/3}\le |\mathcal C^{u_n}_\Max| \le A n^{2/3}]\ge
    1-\varepsilon.
  \end{equation}
\end{theorem}

If $u_n$ is not in the critical window, then the maximal connected
component behaves differently:
\begin{theorem}
  \label{t:outofwindow}
  (a) When $(u_n)_{n\ge 1}$ satisfies
  \begin{equation}
    u_\star-u_n \xrightarrow{n\to\infty}0,
    \qquad\text{and}\qquad
    \omega_n:=n^{1/3}(u_\star-u_n)\xrightarrow{n\to\infty}\infty,
  \end{equation}
  then for
  $v_n= 2 n^{2/3} \omega_n \frac{d-2}{(d-1)^2} e^{-u_\star (d-2)/(d-1)}$
  and for every $\varepsilon >0$
  \begin{equation}
    \Big|\frac{|\mathcal C^{u_n}_\Max|}{v_n}-1\Big| \le \varepsilon
    \qquad \bP_{n,d}\text{-a.a.s.}
  \end{equation}

  (b) When $(u_n)_{n\ge 1}$ satisfies
  \begin{equation}
    u_\star-u_n \xrightarrow{n\to\infty}0,
    \qquad\text{and}\qquad
    \omega_n :=n^{1/3}(u_\star-u_n)\xrightarrow{n\to\infty}-\infty,
  \end{equation}
  then for every $\varepsilon >0$ there exists $B=B(\varepsilon)>0$, such that
  for all $n$ large enough
  \begin{equation}
    \bP_{n,d}
    \big[|\mathcal C^{u_n}_\Max|\le B n^{2/3}|\omega_n|^{-1/2}\big]
    \ge 1-\varepsilon.
  \end{equation}
\end{theorem}

The above theorems confirm that the maximal connected component of the
vacant set behaves similarly as the largest connected cluster of the
Bernoulli percolation on random regular graphs, see
\cite{ABS04,NP10,Pit08}. Remark that the upper bound in part (b) of
Theorem~\ref{t:outofwindow}  seems to be non-optimal. The result is
however sufficient to confirm that the width of the window is $n^{-1/3}$.
To improve such a statement, it is necessary to obtain better estimate in
Theorem~\ref{t:rgds}(ii) that we quote below.

The methods of this paper are largely inspired by the recent article by
Cooper and Frieze \cite{CF10}, where the authors develop a new technique
to prove \eqref{e:supercritical} and \eqref{e:subcritical}. This
technique is very specific to deal with random regular graphs, in
contrast with the results of \cite{CTW11} which hold for a more general
class of graphs, including e.g.~large girth expanders. In the present
article we extend the methods in \cite{CF10} to the critical case. We
believe that obtaining such an extension from the techniques in
\cite{CTW11} should be rather difficult.

The crucial observation of \cite{CF10}, allowing for a very elegant proof
of \eqref{e:supercritical}, \eqref{e:subcritical} for random regular
graphs is the following. Under the annealed measure \eqref{e:annealed},
given the information about the graph discovered by the random walk up to
time $u|V|$, the subgraph of $G$ induced by the vacant set $\mathcal V^u$
is distributed as a random graph uniformly chosen within the set of
graphs with a given (random) degree sequence, see Proposition~\ref{p:CF}
below.

The paper \cite{CF10} further uses the fact that the behaviour of the
uniform random graphs $G_{\boldsymbol d}$ with a given degree sequence
$\boldsymbol d:V\to \mathbb N$ is sufficiently well known. More
precisely, as follows from \cite{MR95}, there exists a single parameter
$Q=Q(\boldsymbol d )$ (see \eqref{e:Q} below) such that $Q<0$ implies
that $G_{\boldsymbol d}$ is typically sub-critical (i.e.~has only small
  components),  and $Q>0$ implies that $G_{\boldsymbol d}$ is
supercritical (i.e.~has a giant component). Moreover, the recent paper
\cite{HM10} establishes the existence of an intermediate regime (the
  so-called critical window) when $Q(\boldsymbol d)$ converges to zero at
a certain rate as $n$ tends to infinity, see also \cite{JL09}. The
results mentioned in this paragraph that will be useful in this paper are
summarised in Theorem~\ref{t:rgds} below.

The principal contribution of this paper is thus to obtain sufficiently
sharp estimates on the random degree sequence of the vacant set
$\mathcal V^u$, and consequently on the value of the parameter $Q$, see
Theorems~\ref{t:Vuvol}, \ref{t:niconc} and \eqref{e:Qest} below. Weaker
estimates of this type were shown in \cite{CF10}, which combined with
\cite{MR95}, allowed them to deduce \eqref{e:supercritical},
\eqref{e:subcritical}.

We should remark that \cite{CF10} contains also a statement on the
critical behaviour. More precisely, Theorem~2(iii) of \cite{CF10} states
that for some $u_n = u_\star (1+o(1))$ (which might be random), the size
of $\mathcal C_{\Max}^{u_n}$ is $n^{2/3+o(1)}$, $\mathbb P_{n,d}$-a.a.s.
Our results considerably improve this statement.

Note also that much more is known about the random graphs with a given
degree sequence, see for instance the results in \cite{FR09} and
\cite{JL09}. Often the hypothesis of these results can be shown to hold
true for $\mathcal{V}^u$, using Theorems~\ref{t:Vuvol} and
\ref{t:niconc}. If this is the case, their conclusions will also apply to
$\mathcal{V}^u$ ($\bP_{n,d}$-a.a.s), providing us with more information
on the geometry of the vacant set.

As an example of such application, we obtain the following improvement on
the statement \eqref{e:supercritical} about the super-critical behaviour
of the vacant set.
\begin{theorem}
  \label{t:supercritical}
  Let $u < u_\star$. Then there is $\rho =\rho (u,d)\in (0,1)$  such that
  for every $\varepsilon >0$
  \begin{equation}
      n^{-1}|\mathcal C^u_\Max|\in
      (\rho -\varepsilon ,\rho +\varepsilon) \qquad
      \bP_{n,d}\text{-a.a.s.}
  \end{equation}
\end{theorem}
In the above statement, the value of $\rho$ can be explicitly calculated,
see \eqref{e:rhodef} below. Remark also that to obtain the above
theorem,  our precise estimates on $Q$ are not necessary, in fact the
precision obtained in \cite{CF10} would have been sufficient.

Finally, let us briefly describe Theorems~\ref{t:Vuvol} and
\ref{t:niconc}. The former, establishes an estimate on the expected
degree distribution of $\mathcal{V}^u$, by approximating the probability
that a random walk visits a neighbourhood of a given vertex $x \in V$
before time $u n$. For this, we make use of the well known relation
between random walks on graphs and discrete potential theory, as well as
the pairing construction introduced by Bollobás, which we detail in
Section~\ref{s:notation}. Then in Theorem~\ref{t:niconc} we prove that
with high probability the degree sequence of $\mathcal{V}^u$ concentrates
around its expectation. This is done using a standard concentration
inequality, together with the fast mixing properties of the random walk
on a random regular graph.

\smallskip

This paper is organised as follows. In Section~\ref{s:notation} we
introduce some of the notation needed in the paper and the pairing
construction of random regular graphs. In Section~\ref{s:prelim}, we
recall the results of \cite{CF10,HM10,JL09} needed later.
Section~\ref{s:rw} contains precise estimates on the behaviour of the
simple random walk on random regular graphs. In
Section~\ref{s:degreesequence}, we give the estimates on the degree
sequence of the vacant set.
Theorems~\ref{t:criticalwindow}--\ref{t:supercritical} are proved in
Section~\ref{s:proofs}. The Appendix summarises some general facts
concerning random walks on finite graphs.
%%%% END OF 01-intro.tex %%%%>>>

%%%% START OF 02-notation.tex %%%%<<<
\section{Notation and definitions}
\label{s:notation}

\subsection{Basic notation}
We now introduce some basic notation. Throughout the text $c$ or $c'$
denote strictly positive constants only depending on $d$, with value
changing from place to place. Dependence of constants on additional
parameters appears in the notation. For instance $c_u$ denotes a positive
constant depending on $u$ and possibly on $d$. We write
$\mathbb N = \{0,1,\dots\}$ for the set of natural numbers, and $[d]$ for
the set $\{1,\dots,d\}$. For a set $A$ we denote by $|A|$ its
cardinality. For any sequence of probability measures $P_n$ and events
$A_n$ we say that $A_n$ holds $P_n$-a.a.s. (asymptotically almost
  surely), when $\lim_{n\to\infty}P_n[A_n]=1$.

In this paper, the term \textit{graph} stands for a finite simple graph,
that is a graph without loops or multiple edges. Sometimes we
intentionally allow the graph to have loops and/or multiple edges and in
this case we use the term \textit{multigraph}. For arbitrary (multi)graph
$G=(V,\mathcal E)$, we use $\dist(\cdot,\cdot)$ to denote the usual graph
distance and write $B(x,r)$ for the closed ball centred at $x$ with
radius $r$, that is $B(x,r)=\{y\in V:\dist(x,y)\le r\}$.

We use $\mathcal G_{n,d}$ (resp. $\mathcal M_{n,d}$) to denote the set of
all $d$-regular graphs (resp.~multigraphs) with vertex set
$V_n=\{1,\dots,n\}$.  Given a degree sequence
$\boldsymbol d:V_n\to \mathbb N$, we use $\mathcal G_{\boldsymbol d}$ to
denote the set of graphs for which every vertex $x\in V_n$ has the degree
$\boldsymbol d_x=\boldsymbol d(x)$. Similarly,
$\mathcal M_{\boldsymbol d}$ stands for the set of such multigraphs; here
loops are counted twice when considering the degree. $\mathbb P_{n,d}$
and $\mathbb P_{\boldsymbol d}$ denote the uniform distributions on
$\mathcal G_{n,d}$ and $\mathcal G_{\boldsymbol d}$ respectively.

\subsection{Pairing construction}
\label{ss:pairing}
In order to study properties of random regular graphs, Bollobás (see
  e.g.~\cite{Bol01}) introduced the so-called pairing construction, which
allows to generate such  graphs starting from a random pairing of a set
with $dn$ elements. The same construction can be used to generate a
random graph chosen uniformly at random from $\mathcal G_{\boldsymbol d}$.
Since this pairing construction will be important in what follows, we
give here a short overview of it.

From now on, whenever we consider a sequence
$\boldsymbol d:V_n\to \mathbb N$, we suppose that
$\sum_{x \in V_n} \boldsymbol d_x$ is even. Given such a sequence, we
associate to every vertex $x \in V_n$, $\boldsymbol d_x$ half-edges. The
set of half-edges is denoted by
$H_{\boldsymbol d} = \{ (x,i): x \in V_n, i \in [\boldsymbol d_x]\}$. We
write $H_{n,d}$ for the case $\boldsymbol d_x=d$ for all $x\in V_n$.
Every perfect matching $M$ of $H_{\boldsymbol d}$ (i.e. partitioning of
  $H_{\boldsymbol d}$ into $|H_{\boldsymbol d}|/2$ disjoint pairs)
corresponds to a multigraph
$G_M = (V_n,\mathcal E_M)\in \mathcal M_{\boldsymbol d}$ with
\begin{equation}
  \label{e:multiedgeset}
  \mathcal E_M = \big\{ \{x,y\}: \big\{(x,i),(y,j)\big\} \in M
    \text{ for some $i \in [\boldsymbol d_x]$,
    $j \in [\boldsymbol d_y]$} \big\}.
\end{equation}
We say that the matching $M$ is simple, if the corresponding multigraph
$G_M$ is simple, that is $G_M$ is a graph. With a slight abuse of
notation, we write $\barP_{\boldsymbol d}$ for the uniform distribution
on the set of all perfect matchings of $H_{\boldsymbol d}$, and also for
the induced distribution on the set of multigraphs
$\mathcal M_{\boldsymbol d}$. It is well known (see e.g.~\cite{Bol01} or
  \cite{McD98}) that a $\barP_{\boldsymbol d}$ distributed multigraph $G$
conditioned on being simple has distribution
$\mathbb P_{\boldsymbol d}$, that is
\begin{equation}
  \label{e:GM}
  \barP_{\boldsymbol d}[G\in \cdot\,|G\in \mathcal
    G_{\boldsymbol d}] = \mathbb P_{\boldsymbol d}[G\in \cdot\,],
\end{equation}
and that, for $d$ constant, there is $c>0$ such that for all $n$ large
enough
\begin{equation}
  \label{e:simple}
  c<
  \barP_{n, d}[G\in \mathcal G_{n, d}]
  <1-c.
\end{equation}
These two claims allow to deduce $\mathbb P_{n,d}$-a.a.s.~statements
directly from $\barP_{n,d}$-a.a.s.~statements.

The  main advantage of dealing with matchings is that they can be
constructed sequentially: To construct a uniformly distributed perfect
matching of $H_{\boldsymbol d}$ one samples \emph{without replacements} a
sequence $h_1,\dots,h_{|H_{\boldsymbol d}|}$ of elements of
$H_{\boldsymbol d}$ in the following way. For $i$ odd, $h_i$ can be
chosen by an arbitrary rule (which might also depend on the previous
  $(h_j)_{j < i}$), while if $i$ is even, $h_i$ must be chosen uniformly
among the remaining half-edges. Then, for every
$1\le i\le |H_{\boldsymbol d}|/2$ one matches $h_{2i}$ with $h_{2i-1}$.

It is clear from the above construction that, conditionally on
$M' \subseteq M$ for a (partial) matching $M'$ of $H_{\boldsymbol d}$,
$M\setminus M'$ is distributed as a uniform perfect matching of
$H_{\boldsymbol d}\setminus\{(x,i):(x,i) \text{ is matched in }M'\}$.
Since the law of the graph $G_M$ does not depend on the labels `$i$' of
the half-edges, we obtain for all partial matchings $M'$ of
$H_{\boldsymbol d}$
\begin{equation}
  \label{e:recursivematchings}
  \barP_{\boldsymbol d}[G_{M\setminus M'}\in \cdot\,|M\supset M'] =
  \barP_{\boldsymbol d'}[G_M\in \cdot],
\end{equation}
where $\boldsymbol d'_x$ is the number of half-edges incident to $x$ in
$H_{\boldsymbol d }$ that are not yet matched in $M'$, that is
$\boldsymbol d'_x=\boldsymbol d_x-
\big|\{\{(y_1,i),(y_2,j)\}\in M':y_1=x,i\in[ \boldsymbol d_x]\} \big|$,
and $G_{M\setminus M'}$ is the graph corresponding to a non-perfect
  matching $M\setminus M'$, defined in the obvious way.

\subsection{Random walk notation}
\label{ss:rwnot}
For an arbitrary multigraph $G=(V,\mathcal E)$, we use $P^G_x$ to denote
the law of canonical continuous-time simple random walk on $G$ started at
$x\in V$, that is of the Markov process with generator given by
\begin{equation}
  \label{e:lap}
  \mathcal L f (x) = \sum_{y \in V} p_{xy}(f(y)-f(x)) ,
  \qquad \text{for } f:V\to \mathbb R, x \in V.
\end{equation}
Here $p_{xy}=n_{xy}/d_x$, $n_{xy}$ is the number of edges connecting $x$
and $y$ in $G$, and $d_x$ is the degree of $x$; the loops are counted
twice in $n_{xx}$ and $d_x$.

We write $P^{G,\ell}_x$ for the restriction of $P^G_x$ to $D([0,\ell],V)$
and $P^{G,\ell}_{xy}$ for the law of random walk bridge, that is for
$P^{G,\ell}_x$ conditioned on $X_\ell=y$. We write
$E^G_x, E^{G,\ell}_x, E^{G,\ell}_{xy}$ for the corresponding
expectations. The canonical shifts on $D([0,\infty),V)$ are denoted by
$\theta_t$. The time of the $n$-th jump is denoted by $\tau_n$, i.e.
$\tau_0=0$ and for $n \geq 1$,
$\tau_n = \inf\{ t \geq 0: X_t \neq X_0\} \circ \theta_{\tau_{n-1}} + \tau_{n-1}$.
The process counting the number of jumps before time $t$ is denoted by
$N_t=\sup\{k:\tau_k\le t\}$. Note that, when $G$ is simple, under $P^G_x$,
$(N_t)_{t \geq 0}$ is a Poisson process on ${\mathbb R}_+$ with intensity
one. We write $\hat X_n$, $n \in \mathbb{Z}_+$, for the discrete skeleton
of the process $X_t$, that is $\hat X_n = X_{\tau_n}$.

Given $A\subset V$, we denote by $H_A$ and $\tilde H_A$ the respective
entrance and hitting time of~$A$
\begin{equation}
  \label{e:hittingtimes}
  \begin{split}
    H_A &= \inf\{t \geq 0 : X_t \in A\},\qquad
    \text{and}\qquad
    {\tilde H}_A = H_A\circ \theta_{\tau_1} + \tau_1.
  \end{split}
\end{equation}

We denote by $\pi$ the stationary distribution for the simple random walk
on $G$, which is uniform if $G$ is $d$-regular (even if $G$ is not
  simple). $P^G$ stands for the law of the simple random walk started at
$\pi $ and $E^G$ for the corresponding expectation. For all real valued
functions $f,g$ on $V$ we define the Dirichlet form
\begin{equation}
  \label{def:dir}
  {\mathcal D}(f,g) =
  \frac{1}{2} \sum_{x,y \in V}
    (f(x)-f(y))(g(x)-g(y)) \pi_x p_{xy}
    = - \sum_{x \in G} \mathcal L f(x) g(x) \pi_x.
\end{equation}
The spectral gap of $G$ is given by
\begin{align}
  \label{def:gap}
  \lambda_G = \min \bigl\{ {\mathcal D}(f,f): \pi(f^2)=1, \pi(f)=0 \bigr\}.
\end{align}
From \cite{Sal97}, p.~328, it follows that for $d$-regular graphs,
\begin{align}
  \label{e:I}
  \sup_{x,y \in V} |P_x[X_t=y] - \pi_y| \leq e^{-\lambda_G t},
  \text{ for all } t \geq 0.
\end{align}
It is also a well known fact (see e.g.~\cite{Fri08}) that there exist
$\alpha >0$ such that
\begin{equation}
  \label{e:gap}
  \lambda_G>\alpha ,\qquad\text{both $\mathbb P_{n,d}$-a.a.s.~and
    $\barP_{n,d}$-a.a.s.}
\end{equation}
%%%% END OF 02-notation.tex %%%%>>>

%%%% START OF 03-preliminaries.tex %%%%<<<
\section{Preliminaries}
\label{s:prelim}
\subsection{Distribution of the vacant set}
\label{ss:distvacset}

Recall the notation $\mathcal V^u=\mathcal V^u_G$ for the vacant set of
the random walk on the graph $G=(V,\mathcal E)$ at level $u$,
\eqref{e:Vu}. For the purpose of this paper, it is suitable to define a
closely related object, the \textit{vacant graph}
$\bbV^u = (V,\mathcal{E}^u)$ where
\begin{equation}
  \label{e:vacantgraph}
\mathcal{E}^u = \{\{x,y\}\in \mathcal E: x,y\in \mathcal V^u_G\}.
\end{equation}
It is important to notice that the vertex set of $\bbV^u$ is $V$ and not
$\mathcal V^u$, in particular $\bbV^u$ is not the graph induced by
$\mathcal{V}^u$ in $G$. Observe however that the maximal connected
component of the vacant set $\mathcal C_\Max$ (defined before in terms of
  the graph induced by $\mathcal{V}^u$ in $G$) coincides with the maximal
connected component of the vacant graph $\bbV^u$ (except when
  $\mathcal V^u$ is empty, but this difference can be ignored in our
  investigations).

We use $\mathcal D^u:V\to \mathbb N$ to denote the (random) degree
sequence of $\bbV^u$, and write $Q_{n,d}^u$ for the distribution of this
sequence under the annealed measure $\bbP_{n,d}$, defined by
$\bbP_{n,d} (\cdot) := \int P^G(\cdot) \barP_{n,d}(\d G)$.

The following proposition from Cooper and Frieze \cite{CF10} allows us to
reduce questions on the properties of the vacant set $\mathcal V^u$ of
the random walk on random regular graphs to questions on random graphs
with given degree sequences.

\begin{proposition}[Lemma 6 of \cite{CF10}]
  \label{p:CF}
  For every $u\ge 0$, the distribution of the vacant graph $\bbV^u$ under
  $\bbP_{n,d}$ is given by $\barP_{\boldsymbol d}$ where $\boldsymbol d $
  is sampled according to $Q^u_{n,d}$, that is
  \begin{equation}
    \label{e:CF}
    \bbP_{n,d}[\bbV^u \in \cdot\,] = \int \bar
    {\mathbb P}_{\boldsymbol d}[G \in \cdot\,] Q_{n,d}^u(\d \boldsymbol d).
  \end{equation}
\end{proposition}

Although a proof of Lemma~\ref{p:CF} can be found in \cite{CF10}, we
provide a proof here for the sake of completeness.

\begin{proof}
  Let $M$ be a $\barP_{n,d}$-distributed pairing of $H_{n,d}$ and let $X$
  be a random walk on $G=G_M$. Define $M_t\subset M$ to be the set of all
  pairs of half-edges incident to a vertex $X_s$ with $s\le t$,
  \begin{equation}
    \label{e:Mt}
    M_t= \big\{\{(x,i),(y,j)\}\in M:x\in \{X_s:s\le t\},i\in[ d] \big\}.
  \end{equation}
  It is easy to see that the edges of the vacant graph $\bbV^u$
  correspond exactly to the pairs in $M\setminus M_{un}$, that is
  $\bbV^u = G_{M\setminus M_{un}}$. In particular, $\mathcal D^u(x)$ is
  the number of the half-edges incident to $x$ not matched in $M_{un}$.
  Denoting by $\mathcal F_u$ the $\sigma $-algebra generated by
  $((X_s,M_s),s\le un)$, the above implies that $\mathcal D^u$ is
  $\mathcal F_u$-measurable.

  It follows from \eqref{e:recursivematchings} that, conditionally on
  $\mathcal F_u$, the distribution of $G_{M\setminus M_{un}}$
  only depends on the sequence of half-edges that are not matched
  in $M_{un}$, and is given by $\barP_{\mathcal D_u}$.
  More precisely,
  \begin{equation}
    \bbP_{n,d}
    [\bbV^u \in \cdot\,|\mathcal F_{u}]
    = \bbP_{n,d}[G_{M\setminus M_{un}}\in
      \cdot\,|\mathcal D^u]=\barP_{\mathcal D^u}[G\in \cdot\,],
  \end{equation}
  and thus
  \begin{equation}
      \bbP_{n,d}[\bbV^u \in \cdot\,]
      =\bbP_{n,d}\big[\bbP_{n,d}
        [\bbV^u \in \cdot\,|\mathcal F_u]\big]
      =\bbP_{n,d}\big[\barP_{\mathcal D^u}
        [G \in \cdot\,]\big]
      =
      \int \bar
      {\mathbb P}_{\boldsymbol d}[G \in \cdot\,] Q_{n,d}^u(\d \boldsymbol
        d),
  \end{equation}
  where the last equality follows from the definition of $Q_{n,d}^u$. This
  concludes the proof of Proposition~\ref{p:CF}.
\end{proof}

\subsection{Behaviour of random graphs with a given degree sequence.}
We now summarise the results about the behaviour of random graphs with
a given degree sequence which will be used in this paper. For a degree
sequence $\boldsymbol d: V_n\to \mathbb N$ we define
\begin{equation}
  \label{e:Q}
  Q(\boldsymbol d)=\frac{\sum_{x=1}^n \boldsymbol d_x^2}
  {\sum_{x=1}^n \boldsymbol d_x}-2,
\end{equation}
and set $n_i(\boldsymbol d)$ to be the number of $x\le n$ with
$\boldsymbol d_x=i$,
\begin{equation}
  n_i(\boldsymbol d) = \big|\{x \in V_n: \boldsymbol d_x = i\}\big|.
\end{equation}
For any graph $G$ we use $\mathcal C_\Max(G)$ to denote the maximal
connected component of $G$.

\smallskip

The following theorem summarises the results of \cite{MR95,JL09,HM10}
needed later.
\begin{theorem}
  \label{t:rgds}
  Let $(\boldsymbol d^n)_{n\ge 1}$, $\boldsymbol d^n:V_n\to \mathbb N$,
  be a sequence of degree sequences. We assume that the degrees are
  uniformly bounded (i.e.~$\max\{\boldsymbol d^n_x:n\ge 1, x\le n\}\le \Delta<\infty$),
  and that $n_1(\boldsymbol d^n)\ge \zeta  n$ for a $\zeta>0$.

  \begin{itemize}
    \item[(i)](critical window) If
    $|Q(\boldsymbol d^n)|\le \lambda n^{-1/3}$ for all $n\ge 1$, then for
    every $\varepsilon >0$ there exists
    $A=A(\zeta ,\lambda ,\varepsilon,\Delta)$  such that
    for all $n$ large enough
    \begin{equation*}
      \barP_{\boldsymbol d^n}[A^{-1}n^{2/3}\le |\mathcal C_{\Max}(G)| \le A
        n^{2/3}]\ge 1- \varepsilon.
    \end{equation*}

    \item[(ii)](below the window) If
    $\lim_{n\to\infty}n^{1/3}Q(\boldsymbol d^n)= -\infty$ and
    $\lim_{n\to\infty}Q(\boldsymbol d^n) = 0$, then for every
    $\varepsilon >0$ exists $B=B(\zeta ,\varepsilon ,\Delta )<\infty$
    such that for all $n$ large enough
    \begin{equation*}
      \barP_{\boldsymbol d^n}
      \Big[|\mathcal C_\Max(G)| < B \sqrt {n/|Q(\boldsymbol d^n)|}\Big]
      > 1-\varepsilon .
    \end{equation*}

    \item[(iii)] (above the window) Let
    $\lim_{n\to\infty}n^{1/3}Q(\boldsymbol d^n)= +\infty$ and
    $\lim_{n\to\infty}Q(\boldsymbol d^n)= 0$.
    In addition, assume that
    \begin{equation}
      \label{e:plimit}
      \lim_{n\to\infty}
      \frac{n_i(\boldsymbol d^n)}{n} = p_i,
      \qquad \text{for all } 0\le i\le \Delta,
    \end{equation}
    for some probability distribution $(p_i)_{0\le i\le \Delta }$ on
    $\{0,\dots,\Delta\}$, and set $\lambda = \sum_{i=0}^\Delta i p_i$,
    $\beta =   \sum_{i=0}^\Delta i(i-1)(i-2) p_i$,
    and $v_n=2n \lambda^2 \beta^{-1} Q(\boldsymbol d^n)$.
    Then, for every $\varepsilon $
    \begin{equation*}
      \Big|\frac{|\mathcal C_\Max(G)|}{v_n}-1\Big|<\varepsilon, \qquad
      \barP_{\boldsymbol d^n}\text{-a.a.s.},
    \end{equation*}

    \item[(iv)] (super-critical regime) Let
    $\lim_{n\to\infty}Q(\boldsymbol d^n)= Q_\infty>0$ and assume that
    \eqref{e:plimit} holds. Let $g$ be the generating function of $(p_i)$,
    $g(x)=\sum_{i=0}^\Delta p_i x^i$. Then there exists a unique solution
    $\xi $ to $g'(x)=\lambda x$ in $(0,1)$, and for $\rho = 1-g(\xi )$
    and any $\varepsilon >0$
    \begin{equation*}
      \Big| \frac{|\mathcal C_\Max(G)|}n - \rho \Big|
      \leq \varepsilon, \qquad \barP_{\boldsymbol d^n}\text{-a.a.s.}
    \end{equation*}
  \end{itemize}
\end{theorem}

\begin{proof}
  Parts (i), (ii) correspond to Theorems 1.1 and 1.2 of \cite{HM10},
  where these statements are proved under more general assumptions. In
  particular, \cite{HM10} does not require the uniform upper bound
  $\Delta$ on the maximal degree. The restriction to the uniformly
  bounded degree sequences implies that the constant $R(\boldsymbol d^n)$
  used in \cite{HM10} satisfies $c<R(\boldsymbol d^n)<c^{-1}$ for all $n$
  large enough and is therefore immaterial for our purposes.

  Parts (iii), (iv) are taken from Theorems 2.3 and 2.4 of \cite{JL09}.
  When reading those theorems it is useful to realise that the
  $(p_i)$-distributed random variable $D$ used in \cite{JL09} satisfies
  $\mathbb E[D]=\lambda $ and that $E[D(D-2)]=\lambda Q_\infty$ in our
  notation.

  Remark however that neither \cite{HM10}, or \cite{JL09} consider degree
  sequences with vertices of degree zero, that is with
  $n_0(\boldsymbol d^n)> 0$. It can however be seen easily, that if
  $n_0(\boldsymbol d^n)$ does not exceed $\zeta 'n$, $\zeta '<1$ (which
    is implied by the assumptions of the theorem), the vertices of degree
  zero do not have any influence on the existence of the giant cluster,
  they only change the constants $A$, $B$ in (i), (ii). For (iii), (iv),
  when $n_0(\boldsymbol d^n)/n\to p_0\neq 0$, one applies the theorem for
  the modified sequence $(\bar{\boldsymbol d}^{n})$ where all vertices of
  degree $0$ are omitted. The new degree sequences
  $\bar{\boldsymbol d}^{n}$ are functions on $V_{\bar n}$, with
  $\bar n = n(1-p_0)+o(1)$. They satisfy
  $n_i(\bar {\boldsymbol d}^{n})/\bar n\xrightarrow{n\to\infty} p_i/(1-p_0)=:\bar p_i$.
  Therefore, denoting by the letters with bars the quantities related to
  the distribution $(\bar p_i)$, we obtain
  $Q(\bar{\boldsymbol d^n})=Q(\boldsymbol d^n)$,
  $\bar \lambda = \lambda /(1-p_0)$, $\bar \beta = \beta /(1-p_0)$,
  $\bar g(x)=(g(x)-p_0)/(1-p_0)$, and $\bar \xi = \xi $. This implies
  that  $\bar v_{\bar n}=v_n$ and $\bar \rho\bar n=\rho n $, confirming
  that zero-degree vertices have no influence on the asymptotics of the
  size of the maximal connected component. This completes the proof.
\end{proof}
%%%% END OF 03-preliminaries.tex %%%%>>>

%%%% START OF 04-rw.tex %%%%<<<
\section{Random walk estimates}
\label{s:rw}
This section contains  estimates on the random walk on random regular
multigraphs which will be useful later in order to estimate the typical
degree sequence of the vacant graph $\bbV^u$.

We start by introducing some notation. We use $\mathbb{T}^d$ to denote
the infinite $d$-regular tree with root $\varnothing$. For a $d$-regular
multigraph $G=(V,\mathcal E)$, a map $\phi$ from $\mathbb{T}^d\to V$ is
said to be a \textit{covering of $G$ from $x \in V$}, if
$\phi(\varnothing) = x$, and for every $y \in \mathbb{T}^d$, $\phi$ maps
the $d$ neighbours of $y$ in $\mathbb T^d$  to the neighbours of
$\phi (y)$ in $G$, including the multiplicities and the loops. For $d$-regular
multigraphs constructed by the pairing construction this means that the
neighbours of $y$ are sent by $\phi $ to the vertices which are paired
with $(\phi (y),i)$, $i\in [d]$.

In agreement with our previous notations, $P^{\mathbb{T}^d}_y$ denotes
the law of the continuous-time simple random walk on $\mathbb{T}^d$
starting from $y$. It is important to notice that fixing a covering $\phi$
from $x$, the image by $\phi$ of a random walk in $\mathbb{T}^d$ with law
$P_\varnothing^{\mathbb{T}^d}$ is distributed as $P^G_x$.

For every finite connected $\mathbb A\subset \mathbb T^d$ and
$z\in \mathbb A$ we now define the escape probabilities
\begin{equation}
  \label{e:eB}
  e_{\mathbb A}(z) = P^{\mathbb{T}^d}_z [\tilde{H}_{\mathbb A} = \infty],
\end{equation}
which can be calculated explicitly in practical examples using the fact that
\begin{equation}
  \label{e:escape}
  P^{\mathbb{T}^d}_y [H_\varnothing = \infty] = \frac{d-2}{d-1},
\end{equation}
for every neighbour $y$ of $\varnothing$. This comes from a standard
calculation for a one-dimensional simple random walk with drift, see
for instance \cite{Woe00}, proof of Lemma~(1.24).

\begin{figure}
%  \psfrag{a}[][][0.6]{$\,\varnothing$}
%  \psfrag{z1}[][][0.5]{$z_1$}
%  \psfrag{z2}[][][0.5]{$z_2$}
%  \psfrag{z3}[][][0.5]{$z_3$}
%  \psfrag{x}[][][0.5]{$\;x$}
%  \psfrag{phi}[][][0.8]{$\phantom{\Big|}^{\phi}$}
%  \psfrag{p1}[][][0.4]{$\;\;\;\;\phi(z_1)$}
%  \psfrag{p2}[][][0.4]{$\;\;\;\;\phi(z_2)$}
%  \psfrag{p3}[][][0.4]{$\;\,\phi(z_3)$}
  \begin{center}
    \includegraphics[width=9cm]{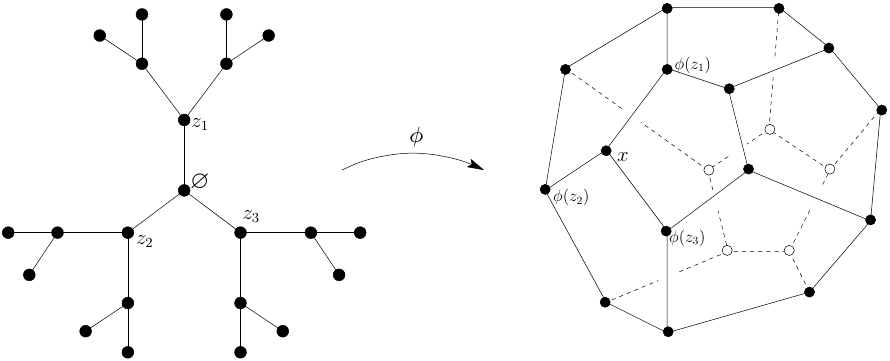}
    \caption{A covering $\phi$ of a regular graph from $x$.}
  \end{center}
\end{figure}

We use $z_i\in \mathbb T^d$, $i\in [d]$ to denote  the neighbours of
$\varnothing$ listed in some predefined order. For $x\in \mathcal V$, let
$\phi_x $ be a covering of $G$ from $x$. From now on, if $G$ was obtained
by the pairing construction, we require that $\phi_x(z_i)$ is the vertex
matched with $(x,i)$. Otherwise $\phi_x$ can be chosen arbitrarily, since
our statements will not depend on which particular choice of $\phi_x$ is
picked. For any $D \subset [d]$, we define the sets
\begin{equation}
  \label{e:BbbB}
  \mathbb B_D=\{\varnothing\}\cup\{z_i: i\in D\}
  \quad\text{and}\quad
  B_{x,D} = \{x\} \cup \{\phi_x(z_i):i\in D\}.
\end{equation}
The sets $B_{x,D}$ will be used later in the calculation of the degree
distribution of the vacant set: using an inclusion-exclusion formula, we
can express the event $\{$the degree of $x$ in $\bbV^u$ is $k\}$ in terms
of events $\{B_{x,D}\subset \mathcal V^u\}$, for $D \subset [d]$.

We first prove a technical lemma describing the graph $V$ after removing
the set $B_{x,D}$.
\begin{lemma}
  \label{l:graphs} For $K>0$, we say that the graph $G$ is \emph{$K$-good},
  if there is no $x\in V$ and $D\subset [d]$ such that:
  \begin{itemize}
    \item[(a)] $B(x,3)$ is a tree, and
    \item[(b)] either $V\setminus B_{x,D}$ is disconnected or
    $\diam (V\setminus B_{x,D})\ge K \log n$.
  \end{itemize}
    Then, there is $K>1$ such
  that $G$ is $K$-good, $\barP_{n,d}$-a.a.s.
\end{lemma}
\begin{proof}
  To prove this claim we use Lemma~2.14 of \cite{Wor99} which states that
  $\barP_{n,d}$-probability that there exists a $(s,j)$-separating set in
  $G$, that is a set $S\subset V$ with $|S|=s$ such that $G\setminus S$
  contains a component of exactly $j$ vertices, satisfies
  \begin{equation}
    \label{e:separating}
    \barP_{n,d}[G \text{ has an $(s,j)$-separating set}]
    \le 3^{2+s/d}\Big(\frac{j+s}n\Big)^{j(\frac d2-1)} n^{\frac s2}
    (j+s)^{\frac32 s}.
  \end{equation}
  (This formula is used in \cite{Wor99} to prove that $G$ is
    $\mathbb P_{n,d}$-a.a.s.~$d$-connected. However, the $d$-connectedness
    cannot be used directly to prove our claim, since $|B(x,1)|=d+1$.)

  Observe first that it is sufficient to consider $D=[d]$. Indeed, if
    $x$ is such that $B(x,3)$ is a tree and the graph is connected after
  removing $B_{x,[d]}$, then it is connected after removing $B_{x,D}$ for
  any $D\subset [d]$.  In addition, obviously
  $\diam (G\setminus B_{x,D})\le \diam(G\setminus B_{x,[d]}) +2$.

  Let now $D=[d]$, that is $B_{x,D}=B(x,1)$. Assume that
  $G\setminus B(x,1)$ is disconnected. Then removing the set
  $B(x,1)\setminus \{x\}$ from $G$, divides the graph into at least three
  components. One of them is $\{x\}$, and the other ones are are
  contained in $G\setminus B(x,1)$. Since we require that $B(x,3)$ is a
  tree and since $G$ is $d$-regular, the size of these other components
  is at least $4$. We thus apply \eqref{e:separating} with $s=d$ and
  $j\ge 4$: The $\barP_{n,d}$-probability that there is $x\in V$ such
  that $B(x,3)$ is a tree and $G\setminus B(x,1)$ is disconnected is
  bounded from above by
  \begin{equation}
    \label{e:suma}
    \sum_{j= 4}^{n/2} \barP_{n,d}[G \text{ has an $(s,j)$-separating
        set}].
  \end{equation}
  The largest term in this sum, corresponding to $j=4$, is
  $O(n^{4-\frac 32 d})=o(1)$. All the remaining terms are much smaller.
  Actually, as in the proof of Theorem~2.10 of \cite{Wor99}, it can be
  shown that \eqref{e:suma} tends to $0$ as $n\to\infty$. Hence,
  $\barP_{n,d}$-a.a.s.~there is no $x\in V$ such that $B(x,3)$ is a tree
  and $G\setminus B(x,1)$ is disconnected.

  To bound the diameter of $G\setminus B(0,1)$ we use the fact that
  $\barP_{n,d}$-a.a.s.~$\diam(G)\le 2\log_{d-1}n$, see~\cite{BF82}. We
  may also assume that $G\setminus B(x,1)$ is connected, which as we
  proved occurs $\barP_{n,d}$-a.a.s.

  We now claim that
  \begin{equation}
    \label{e:remove}
    \parbox{0.85\textwidth}{removing one vertex $v$ of degree $d$ in an
      arbitrary graph while keeping it connected can increase the
      diameter of the graph at most by a factor of~$3^{d-1}$.}
  \end{equation}
  To prove this claim we first consider the removal of an edge: Removing
  one edge $e$ from a graph while keeping in connected can increase the
  diameter of the graph at most by factor $3$. To see this it is
  sufficient to consider the shortest path $\mu $ in $G\setminus\{e\}$
  connecting the vertices of $e$ (such path must exist since
    $G\setminus\{e\}$ is connected, and cannot be longer than $2\diam G$),
  and to replace the edge $e$ by the path $\mu $ in every geodesic of $G$
  that contains $e$.

  Having understood the removal of edges, we can analyse the removal of
  a vertex $v$. We first remove all but one of the edges of $G$ incident
  to $v$, in this procedure the diameter of the graph is  multiplied by
  at most $3^{d-1}$. Removing $v$ together with the last edge linking it
  to $G\setminus \{v\}$ yields the claim \eqref{e:remove}.

  The claim \eqref{e:remove} and $|B(x,1)|\le d+1$ then imply that
  $\barP_{n,d}$-a.a.s
  \begin{equation}
    \diam(G\setminus B(x,1))\le 3^{(d-1)(d+1)}\diam G\le
    3^{(d-1)(d+1)}\cdot 2 \log_{d-1}n.
  \end{equation}
  This completes the proof of the lemma.
\end{proof}

We now start controlling how the random walk visits the sets $B_{x,D}$
defined in \eqref{e:BbbB}. In the lemma below, we show that the
probability of escaping from $B_{x,D}$ for a large time can be
approximated by the escape probability on the infinite tree, defined in
\eqref{e:eB}.

\begin{lemma}
  \label{l:bbE}
  For every $x \in V_n$,
  $D \subset [d]$ and $i \in D$
  \begin{equation}
    \label{e:bbE}
    \barE_{n,d} \Big[ \Big| P^G_{\phi_x(z_i)} [ \tilde{H}_{B_{x,D}} > \log^2 n]
      - e_{\mathbb{B}_D}(z_i) \Big| \Big] \leq \frac{c \log^4 n}{n}.
  \end{equation}
  and
  \begin{equation}
    \label{e:hitx}
    \bbP_{n,d} [ \tilde{H}_{B_{x,D}} \leq \log^2 n]  \leq \frac{c \log^2 n}{n}.
  \end{equation}
\end{lemma}

\begin{proof}
  To simplify the notation we write $B$, $\mathbb B$, $z$ and $\phi (z)$  for
  $B_D$, $\mathbb B_{x,D}$, $z_i$ and $\phi_x(z_i)$. Using the fact that
  $\phi_x$ maps a random walk on $\mathbb T^d$ to a random walk on $G$,
  we can write
  \begin{equation}
    \begin{split}
      P^G_{\phi (z)} [ &\tilde{H}_{B} \leq \log^2 n]
      = P^{\mathbb{T}^d}_z [ \tilde{H}_{\phi_x^{-1}(B)} \leq \log^2 n]
      \\&= P^{\mathbb{T}^d}_z [ \tilde{H}_{\mathbb{B}} \leq \log^2 n]
       + P^{\mathbb{T}^d}_z [ \tilde{H}_{\phi_x^{-1}(B)\setminus \mathbb{B}}
        \leq \log^2 n, \tilde{H}_{\mathbb{B}} > \log^2 n].
    \end{split}
  \end{equation}
  Therefore, the left-hand side of \eqref{e:bbE} can be bounded from above
  by
  \begin{equation}
    \label{e:splithit0}
    \Big| P^{\mathbb{T}^d}_z [ \tilde{H}_{\mathbb B} > \log^2 n] - e_{\mathbb{B}}(z) \Big|
    + \barE_{n,d} \Big[ P^{\mathbb{T}^d}_z
      [ \tilde{H}_{\phi_x^{-1}(B)\setminus \mathbb{B}} \leq \log^2 n] \Big].
  \end{equation}
  Using the Markov property at time $\log^2 n$ and the definition of
  $e_{\mathbb{B}}(z)$, the first term equals
  \begin{equation}
    \begin{split}
      &P^{\mathbb{T}^d}_z [\log^2 n < \tilde{H}_{\mathbb B} < \infty]
      \\&\leq P^{\mathbb{T}^d}_z \Big[ d(\varnothing, X_{\log^2 n})
        \leq \frac{d-2}{2d} \log^{2} n \Big]
       + \sup_{u: d(u,\varnothing) > \frac{d-2}{2d} \log^{2} n}
      P^{\mathbb{T}^d}_u \Big[ H_\varnothing < \infty \Big].
    \end{split}
  \end{equation}
  Since $d(X_t,\varnothing)$ under $P^{\mathbb{T}^d}_z$ is a random walk on
  $\mathbb{N}$ with expected drift given by $(d-2)/d$, both terms above are
  bounded by $c \exp\{-c' \log^2 n\}$.

  To bound the second term in \eqref{e:splithit0}, note that if at time $t$
  the random walk on $\mathbb T_d$ started from $z$ visits a point in
  $\phi_x^{-1}(B_{x,D})\setminus \mathbb{B}$, then the trajectory of the
  image walk on $G$ together with $B(\phi(z),2)$ contains a cycle in $G$.
  Therefore, denoting by
  $G|_A$ the subgraph of $G$ generated by $A\subset V_n$,
  \begin{equation}
    \barE_{n,d} \Big[ P^{\mathbb{T}^d}_z
      [ \tilde{H}_{\phi_x^{-1}(B)\setminus \mathbb{B}} \leq \log^2 n] \Big]
    \le
    \bbP_{n,d}\Big[\text{$G|_{ B(X_0,2)\cup \{X_s:s\le \log^2 n\}}$
        contains a cycle} \Big].
  \end{equation}
  Taking care of the possibility that the continuous-time random walk
  makes more than $2\log^2 n $ steps before time $\log^2 n$, using the
  notation from Section~\ref{ss:rwnot}, this is
  bounded from above by
  \begin{equation}
    \label{e:mnb}
    \begin{split}
      &P_z^{\mathbb T^d}[N_{\log^2 n}\ge 2\log^2 n]
      +
      \bbP_{n,d}\Big[\text{$G|_{ B(X_0,2)\cup \{\hat X_i:i\le 2 \log^2 n\}}$  contains a cycle} \Big]
        \end{split}
  \end{equation}
  The random variable $N_{\log^2 n}$ has Poisson distribution with mean
  $\log^2 n$, therefore the first term in \eqref{e:mnb} is smaller than
  $ c e^{-c'\log^2 n}$. To bound the second term, observe that the
  considered subgraph can be constructed inductively by the following
  variant of the construction from Section~\ref{ss:pairing}:
  \begin{enumerate}
    \item Choose $\hat X_0$ uniformly at random in $V_n$. Use the pairing
    construction of Section~\ref{ss:pairing} to construct the set
    $B(\hat X_0,2)$. This requires creating at most $d+d(d-1)$ pairs.

    \item Let $\hat X_1$ be a uniformly chosen neighbour of $\hat X_0$.
    (All neighbours of $X_0$ are known from the first step.)

    \item Repeat for all $i\in\{2,\dots,2 \log^2 n\}$ the following steps:
    \begin{enumerate}
      \item Choose $Z_i$ uniformly in $[d]$, independently of the
      previous randomness.
      \item If the half-edge $(\hat X_{i-1},Z_i)$ is not yet matched, then
      match it with an half-edge chosen uniformly among the remaining
      half-edges, as in the pairing construction.
      \item Let $\hat X_i$ be the vertex that is matched with the
      half-edge $(\hat X_{i-1},Z_i)$.
    \end{enumerate}
  \end{enumerate}
  The probability that a cycle is created in the step (1) is easily
  bounded by $c/n$. It is not possible to create any cycle in the step (2).
  In the step (3) the cycle is created only when the half-edge
  $(\hat X_{i-1},Z_i)$ is not yet matched (otherwise we do not add any new
    edge to the subgraph) and when the vertex $\hat X_i$ was
  already `visited by the algorithm', that is it has some matched
  half-edges from the previous steps. Since the algorithm visits at most
  $d+d(d-1)+2\log^2 n$ vertices, the probability to match
  $(\hat X_{i-1},Z_i)$  with an already visited vertex is smaller than
  $cn^{-1} \log^2 n$. Therefore the probability to create the cycle in the
  step (3) is at most $cn^{-1} \log^4 n$. In consequence, the second term
  in \eqref{e:splithit0} is smaller than
  $c e^{-c'\log^2 n}+c n^{-1} \log^4 n$
  which establishes \eqref{e:bbE}.

  In order to prove \eqref{e:hitx}, observe that
  \begin{equation}
      \bbP_{n,d}   [ H_{B} \leq \log^2 n] \le
      \bbP_{n,d} [N_{\log^2 n}\ge 2\log^2 n]+
      \sum_{i=0}^{2\log^2 n}\bbP_{n,d}\Big[
        \hat X_i\in B \Big].
  \end{equation}
  Since under $\bbP_{n,d}$, the random vertex $\hat X_i$ is uniformly
  distributed in $V_n$,
  $\bbP_{n,d}[\hat X_i\in B]\le n^{-1} |B| \le (d+1)n^{-1}$. Hence,
  \begin{equation}
      \bbP_{n,d}   [ H_{B} \leq \log^2 n]
      \leq c' e^{-c \log^2 n} +2 (d+1)n^{-1} \log^2 n  \leq c n^{-1} \log^2 n.
  \end{equation}
  This completes the proof of Lemma~\ref{l:bbE}.
\end{proof}

The previous lemma has a simple corollary.
\begin{corollary}
  \label{c:escape}
  Fix $D \subset [ d]$ and $i \in D$, then both
  $\barP_{n,d}$-a.a.s.~and $\mathbb P_{n,d}$-a.a.s.
  \begin{equation}
    \Big|\big\{
      x \in V_n: \big|P^G_{\phi_x(z_i)}[\tilde{H}_{B_{x,D}} > \log^2 n]
      - e_{\mathbb B_{D}}(z_i) \big| > n^{-1/2}
      \big\}\Big|
    \leq (\log^5 n) n^{1/2}
  \end{equation}
  and
  \begin{equation}
    \Big|\big\{x \in V_n: P^G[H_{B_{x,D}} \leq \log^2 n]
      > n^{-1/2} \big\}\Big|
    \leq (\log^3 n) n^{1/2}.
  \end{equation}
\end{corollary}

\begin{proof}
  Note that the complements of the above events are respectively contained
  in the events
  $\sum_{x\in V_n} |P^G_{\phi_x(z_i)}[\tilde{H}_{B_{x,D}} \leq \log^2 n] - e_{\mathbb{B}}(z_i)| > \log^5 n$
  and $\sum_{x\in V_n} P[H_{B_{x,D}} \leq \log^2 n] > \log^3 n$. Thus,
  the $\barP_{n,d}$-a.a.s.~statements of the corollary are implied by the
  Markov inequality and  Lemma~\ref{l:bbE}. The
  $\mathbb P_{n,d}$-a.a.s.~statements then follow using the
  remark below \eqref{e:simple}.
\end{proof}

In what follows we will make use of the quasi-stationary distribution
with respect to a set $B \subset V$. The usual definition of this
distribution  is given in the Appendix, see \eqref{e:sigmavone}. The
quasi-stationary distribution can be thought as the asymptotic
distribution of the position of a random walk conditioned not to visit $B$.
To proceed, we will need the following lemma which gives the rate of
convergence of such conditioned walk towards the quasi-stationary
distribution.

\begin{lemma}
  \label{l:quasi}
  Let $t_n=\log^2n$. Then, $\barP_{n,d}$-a.a.s., for any $x$ such that
  $B(x,3)$ is a tree and for any connected $A\subset B(x,1)$
  \begin{equation}
    \label{e:quasi}
    \sup_{x, y \in {V} \setminus A} \Big| P^G_x[Y_{t_n} = y| H_A > t_n] -
    \sigma_A(y) \Big| \leq c e^{-c' \log^2 n}.
  \end{equation}
\end{lemma}

\begin{proof}
  By Lemma~\ref{l:graphs}, $V\setminus A$ is $\barP_{n,d}$-a.a.s.~connected
  and has diameter smaller than $K\log n$. To prove the claim of the
  lemma we are going to make use of Lemmas~\ref{l:AB} and
  \ref{l:quasigen}, which are general results on Markov chains presented
  in the Appendix. We use the notation introduced there, in particular
  $0\le\lambda_1^A\le \lambda_2^A\le \dots$ stand for the eigenvalues of
  the $(-\mathcal L^A)$ where $\mathcal L^A$ is the generator of the
  random walk killed on hitting $A$.

  From \eqref{e:gap}, $\lambda_G\ge \alpha >0$,   $\mathbb P_{n,d}$-a.a.s.
  Therefore, using Lemma~\ref{l:AB}, we obtain that
  \begin{equation}
    \lambda^A_{2} - \lambda^A_{1} \geq \alpha  - \frac{1}{E[H_A]},
    \qquad \mathbb P_{n,d}-\text{a.a.s.},
  \end{equation}

  To obtain  an upper bound on $E[H_A]^{-1}$, we use Proposition~3.1 and
  (2.10) of \cite{CTW11}. Using this proposition with $A = A$ and
  $C = V \setminus A$, we obtain
  \begin{equation}
    \label{e:crudeEHB}
    \frac{1}{E[H_A]} \le c \pi(A) \leq \frac{c}{n}.
  \end{equation}
  This implies that $\mathbb
  P_{n,d}$-a.a.s.~$\lambda^A_{2} - \lambda^A_{1} \ge c > 0$.

  Using the above fact together with Lemma~\ref{l:quasigen} (observing
    that $\pi(x)=1/n$ on regular graphs) we obtain that $\barP_{n,d}$-a.a.s.
  \begin{equation}
    \label{e:bdwithsigma}
    \sup_{x, y \in {V} \setminus A} \big| P^G_x[Y_{t_n} = y| H_A > {t_n}]
    - \sigma_A(y) \big|
    \leq \frac {c n^{3/2} \exp\{- c \log^2 n\} }
    {\inf_{z \in V\setminus A}\sigma_A(z)}.
  \end{equation}

  We have to bound the infimum in the denominator.
  For this, take $z \in V \setminus A$ and any $t \geq 0$. By
  reversibility, for any $z' \in V \setminus A$,
  \begin{equation}
    \label{e:zzprime}
    P^G_{z'} [X_t = z | H_A > t]
    = P^G_{z} [X_t = z' | H_A > t]
    \frac{P^G_z[H_A > t]}{P^G_{z'}[H_A > t]}.
  \end{equation}
  In order to bound the above ratio, note that
  \begin{equation}
    P^G_z[H_A > t] \geq P^G_z[H_{z'} < H_A, H_A \circ \theta_{H_{z'}} > t]
    = P^G_z[H_{z'} < H_A] P^G_{z'}[H_A > t].
  \end{equation}
  As $\barP_{n,d}$-a.a.s.~the graph induced by $V \setminus A$ has
  diameter at most $K \log n$, we can find a path $\gamma$ with length at
  most $K \log n$, connecting $z$ and $z'$ and not passing through $A$.
  This gives us that $P_z[H_{z'} < H_A] \ge d^{-c \log n} \ge c n^{-c'}$.
  From Lemma~\ref{l:quasigen}
  $\lim_{t \to \infty} P_w[X_t = z| H_A > t] = \sigma_A(z)$ uniformly for
  all $w,z \in V \setminus A$. Therefore, taking the limit $t\to\infty$
  in \eqref{e:zzprime}, $\sigma_A(z)\ge c \sigma_A(z') n^{-c'}$. Together
  with the fact that $\sigma_A$ is a probability measure this yields
  \begin{equation}
    \inf_{z \in V \setminus A} \sigma_A(z) \geq c n^{-c'}.
  \end{equation}
  Using the above result with \eqref{e:bdwithsigma} finishes the proof
  of Lemma~\ref{l:quasi}.
\end{proof}

\section{Degree sequence of the vacant graph}
\label{s:degreesequence}

We are now in position to estimate the typical degree sequence of the
vacant graph $\mathcal{V}^u_G$ under $\bbP_{n,d}$. At this point it is
instructive to mention the relation between this set and the random
interlacements process on the $d$-regular tree $\mathbb{T}^d$.

The model of random interlacements on transient weighted graphs is
constructed in \cite{Tei09}, see also \cite{Szn10} for the original
construction of the model in the particular case of $\mathbb{Z}^d$,
$d \geq 3$. Random interlacement on $\mathbb T^d$ can be understood as  a
measure $Q^u$ on the space $\{0,1\}^{\mathbb{T}^d}$ which samples a
random subset $\mathcal{V}^u_{{\mathbb T^d}}$ of $\mathbb{T}^d$ (called
  the vacant set left by random interlacements at level $u$)
characterised by the following:
\begin{equation}
  \label{e:charQu}
  Q^u [ K \subset \mathcal V^u_{\mathbb T^d}]
  = \exp \{ -u \, \capacity(K)\},\qquad
  \text{ for every finite $K \subset \mathbb{T}^d$},
\end{equation}
where $\capacity(K) = \sum_{x \in  K} e_{K} (x)$
for $e_{K}$ as in \eqref{e:eB}, cf.~(2.27) of \cite{Tei09}.

Intuitively speaking, the vacant set of the random interlacement
$\mathcal{V}^u_{\mathbb T^d}$ gives the asymptotic local picture of
$\mathcal{V}^u_G$ under $\bP_{n,d}$ as $n$ tends to infinity, see
Proposition~6.3 of \cite{CTW11}. An important fact about random
interlacements on $\mathbb T^d$ is that the law of the vacant cluster of
the root $\varnothing \in \mathbb{T}^d$ under $Q^u$ is the same as the
law of a certain (inhomogeneous) Galton-Watson tree. More precisely, the
probability that $\varnothing$ is vacant equals
$e^{-u\capacity(\varnothing)}=\exp\{-u\frac{d-2}{d-1}\}$.  Given that
$\varnothing\in \mathcal V^u_{\mathbb T^d}$, the offspring distribution
of the root is binomial with parameters $d$ and
\begin{equation}
  \label{e:pu}
  p_u=\exp\Big\{-u\frac{(d-2)^2}{d(d-1)}\Big\},
\end{equation}
and the offspring distribution of all remaining individuals is binomial
with parameters $d-1$ and $p_u$.
Using this characterisation it is easy to compute the probability  that
the degree of the root in $\mathcal V^u_{\mathbb T^d}$ is $i$, $i=0,\dots, d$.
\begin{equation}
  \begin{split}
    \label{e:dui}
    d^u_i&:=Q^u[\text{degree of $\varnothing$ in
    $\mathcal V^u_{\mathbb T^d}$ equals $i$}]
    \\&=Q^u[\text{$\varnothing$ is vacant and has exactly $i$ offsprings}]
    \\&=e^{-u\frac{d-2}{d-1}}\binom d i p_u^i (1-p_u)^{d-i}.
  \end{split}
\end{equation}

We now explain the relation between $d^u_i$ and the degree sequence of
$\mathcal{V}^u_G$. This relation was already obtained in a weaker form in
Theorem~3 of \cite{CF10} and could also be extracted from \cite{CTW11}
Proposition~6.3. However the finer control of errors obtained in
Theorem~\ref{t:Vuvol} is crucial if one wants to stay inside the critical
window.

Recall from Section~\ref{ss:distvacset} that $\mathcal D^u$ denotes the
degree sequence of the vacant graph $\bbV^u$ and that, for any
degree sequence $\boldsymbol d$, $n_i(\boldsymbol d)$ denotes the number
of vertices with degree $i$ in $\boldsymbol d$.
\begin{theorem}
  \label{t:Vuvol}
  For every $u>0$ and every $i\in\{0,\dots,d\}$,
  \begin{equation}
    \big|E^G[n_i(\mathcal D^u)] - n d_i^u \big| \leq c (\log^5 n) n^{1/2},
    \quad
    \barP_{n,d}\text{-a.a.s}.
  \end{equation}
\end{theorem}

\begin{proof}
  We fix $x \in V_n$, $D \subset[d]$, and recall from Section~\ref{s:rw} the
  definitions of the covering map $\phi $, of $B_{x,D}$, $\mathbb B_D$ and
  of $z_i$. To simplify
  the notation we use $\sigma $ and $B$ as shorthand for
  $\sigma_{B_{x,D}}$ and $B_{x,D}$. We first estimate the probability
  that $B\subset \mathcal V^u$. Using the Markov property and
  Lemma~\ref{l:quasi}, we obtain that $\mathbb P_{n,d}$-a.a.s.
  \begin{equation}
    \begin{split}
      \label{e:BinV}
      &\Big|P^G [H_{B} > un]
      - \exp\Big\{-\frac{un}{E^G_\sigma[H_{B}]}\Big\} \Big|
      \\ &= \Big| P^G [H_{B} > {t_n}]
      E^G\big[ P^G_{X_{{t_n}}}[H_{B} > un - {t_n}]
        \big| H_{B} > {t_n}\big]
      - \exp\Big\{-\frac{un}{E^G_\sigma[H_{B}]}\Big\} \Big|
      \\ &\le \Big| P^G [H_{B} > {t_n}]
      P^G_\sigma[H_{B} > un - {t_n}]
      - \exp\Big\{-\frac{un}{E^G_\sigma[H_{B}]}\Big\} \Big|
      + c e^{-c' \log^2 n}.
    \end{split}
  \end{equation}
  Under the measure $P_\sigma$, the random variable
  $H_{B}$ is exponentially distributed, see for instance \cite{AB93}
  below (12). Hence, using that $e^{-t}$ is $1$-Lipschitz for $t\ge 0$,
  \begin{equation}
    \Big|P^G_\sigma[H_{B} > un - {t_n}]
    - \exp\Big\{-\frac{un}{ E_\sigma[H_{B}]}\Big\}\Big| \le
    \frac{t_n}{E^G_\sigma [H_{B}]}.
  \end{equation}
  Therefore, \eqref{e:BinV} becomes
  \begin{equation}
    \begin{split}
      \label{e:BinVb}
      \Big|P^G [H_{B} > un]
      - \exp\Big\{-\frac{un}{E^G_\sigma[H_{B}]}\Big\} \Big|
      \le c e^{-c' \log^2 n} + \frac{t_n}{E_\sigma[H_{B}]}
      + P [H_{B} \le {t_n}].
    \end{split}
  \end{equation}

  Let $V_\good\subset V$ be the set of vertices $x \in V$ satisfying
  \begin{equation}
    \begin{array}{c}
      \label{e:goodx}
      \text{$B(x,2)$ is a tree, and for every $D \subset [d]$
        and $i \in D$,}\\ \big|P^G_{\phi (z_i)}[\tilde{H}_{B_{x,D}} > t_n] -
      e_{\mathbb{B}}(\phi(z_i)) \big| \leq n^{-1/2} \text{ and } P^G[H_{B_{x,D}}
        \leq t_n] \leq n^{-1/2}.
    \end{array}
  \end{equation}
  By Corollary~\ref{c:escape} above, and by Remark~1.4
  and Lemma~6.1 of \cite{CTW11}, the complement of $V_\good$ is very
  small,
  \begin{equation}
    \label{e:Vgoodbig}
    |V\setminus V_\good|\le c n^{1/2}(\log n)^5, \qquad
    \mathbb P_{n,d}\text{-a.a.s.}
  \end{equation}

  By Lemma~2 of \cite{AB93} and \eqref{e:crudeEHB}
  \begin{equation}
    \label{e:invE}
    {E^G_\sigma[H_{B}]}^{-1} \leq
    {E^G[H_{B}]}^{-1} \leq cn^{-1}.
  \end{equation}
  Therefore for $x\in V_\good$, \eqref{e:BinVb} becomes
  \begin{equation}
    \label{e:PEsigma}
    \Big|P^G [H_{B} > un] -
    \exp\Big\{-\frac{un}{E^G_\sigma[H_{B}]}\Big\} \Big|
    \leq c n^{-1/2}.
  \end{equation}

  Our next step is to obtain an estimate on $E_\sigma[H_{B}]$ for
  $x\in V_\good$. To this aim we `collapse' the set
  $B_{x,D}$ into one point $b$, and define a new Markov chain whose
  distribution is denoted $\bar P$  and which is characterised by its
  transition rates $\bar p_{xy}$,
  \begin{equation}
    \label{e:barP}
    \begin{cases}
      \bar p_{ww'} = p_{ww'}, & \text{ if $w, w' \neq b$}, \\
      \bar p_{wb} = \sum_{y \in B} p_{wy}, & \text{ if $w \neq b$}, \\
      \bar p_{bw} = \frac{1}{|B|} \sum_{y \in B} p_{yw}, & \text{ if
      $w \neq b$}.
    \end{cases}
  \end{equation}
  It is easy to check that
  $\bar \pi = \frac 1n \big(|B| \delta_b + \sum_{x \neq b} \delta_x\big)$
  is a reversible distribution for this chain. Therefore,
  \begin{equation}
    \label{e:stationary}
    \begin{split}
      \frac n{|B|} = \bar{E}_b[\tilde{H}_b]
      & = \bar E_b[\tilde{H}_b, \tilde{H}_b \le t_n]
      + \bar P_b[\tilde{H}_b > t_n]
      \bar E_b\big[\bar E_{X_{t_n}}[H_b - t_n] \big| \tilde{H}_b > t_n\big].
    \end{split}
  \end{equation}
  By \eqref{e:barP},
  $\bar P_b[\tilde{H}_b > t_n] = {|B|}^{-1} \sum_{y \in B} P^G_y[\tilde{H}_{B} > t_n]$.
  Therefore, using Lemma~\ref{l:quasi} and \eqref{e:stationary},
  \begin{equation}
    \Big| \sum_{y \in B} P^G_y[\tilde{H}_{B} > t_n]
    E^G_\sigma[H_{B}] - n \Big| \leq ct_n + c\exp\{-c' t_n\} \leq c
    t_n.
  \end{equation}
  Using \eqref{e:invE}, this yields the following estimate on
  $E^G_\sigma[H_B]$,
  \begin{equation}
    \label{e:EsigmaP}
    \Big| \sum_{y \in B} P^G_y[\tilde{H}_{B} > t_n]
    - \frac{n}{E^G_\sigma[H_{B}]} \Big|
    \leq \frac{c t_n}{E^G_\sigma[H_{B}]} + c\exp\{-c' t_n\} \leq \frac{c\log^2 n}{n}.
  \end{equation}

  We are now in position to give our final estimate on $P^G[H_B\ge un]$.
  By the triangle inequality, for $x\in V_\good$,
  \begin{equation}
    \label{e:vacantB}
    \begin{split}
      \Big| P^G[&H_{B} > un]  - \exp
      \Big\{-u \sum_{y \in \mathbb{B}} e_{\mathbb{B}}(y) \Big\} \Big|
      \\&\le \Big|P^G [H_{B} > un]
      - \exp\Big\{-\frac{un}{E_\sigma[H_{B}]}\Big\} \Big|
      \\&\quad +  \Big|\exp\Big\{-\frac{un}{E_\sigma[H_x]}\Big\} -
      \exp\Big\{-u  \sum_{y \in B}
        P^G_y[\tilde{H}_{B_{x,D}} > t_n] \Big\}\Big|
      \\& \quad +  \Big|\exp\Big\{-u \sum_{y \in B_{x,D}}
        P^G_y[\tilde{H}_{B_{x,D}} > t_n] \Big\} -
      \exp\Big\{-u \sum_{y \in \mathbb{B}} e_{\mathbb{B}}(y) \Big\} \Big|
      \\& \leq c n^{-1/2},
    \end{split}
  \end{equation}
  where for the last inequality we used the estimates \eqref{e:PEsigma},
  \eqref{e:EsigmaP} and \eqref{e:goodx}.

  If $x \in V_\good$, all its neighbours are distinct.  We can then use
  the inclusion-exclusion formula to write
  \begin{equation}
    \begin{split}
      P^G \big[&\mathcal D^u(x)=i\big]
      = \sum_{C \subset [d],|C| = i}
      P^G[\mathcal V^u\cap B(x,1)=B_{x,C}]
      \\& = \sum_{C \subset [d],|C| = i}\,\,
      \sum_{D \subset [d], C \subset D}
      (-1)^{|D| - |C|} P^G \big[ H_{B_{x,D}}> un \big].
    \end{split}
  \end{equation}
  Using \eqref{e:vacantB}, we obtain that
  \begin{equation}
    \label{e:vacantC}
    \begin{split}
    \bigg| P^G \big[&\mathcal D^u(x)=i\big] -  \sum_{C \subset [d],|C| = i} \,\,
      \sum_{D \subset [d], C \subset D}
      (-1)^{|D| - |C|}
      \exp\Big\{-u\sum_{y\in \mathbb B_D}e_{\mathbb B_D}(y)\Big\} \bigg|
      \leq cn^{-1/2}.
      \end{split}
  \end{equation}
  From \eqref{e:eB},\eqref{e:escape}, it is not difficult to deduce that
  for $y\in B_{x,D}\setminus \{x\}$,
  \begin{equation}
    e_{\mathbb B_D}(y)=\frac{d-1}{d}\cdot\frac{d-2}{d-1} \qquad
    \text{and}\qquad
    e_{\mathbb B_D}(x)=\frac{d-|D|}{d}\cdot \frac{d-2}{d-1}.
  \end{equation}
  Inserting this into \eqref{e:vacantC} leads to
  \begin{equation}
  \begin{split}
    \bigg| P^G \big[&\mathcal D^u(x)=i\big] -   \binom  d i \sum_{j=1}^d (-1)^{j-i}\binom {d-i}{j-i}
    \exp\Big\{- u \frac{d-2}{d-1}
      \Big(j \frac{d-1}{d} + \frac{d-j}{d}\Big)\Big\} \bigg|
       \leq c n^{-1/2}.
      \end{split}
  \end{equation}
  A simple computation implies that the leading term in the last formula
  equals $d^u_i$ (see \eqref{e:dui}). Therefore, $\barP_{n,d}$-a.a.s.,
  uniformly for $x\in V_\good$,
  \begin{equation}
   \big| P^G \big[\mathcal D^u(x)=i\big]
    -d^n_i\big|\le cn^{-1/2}.
  \end{equation}
  The claim of Theorem~\ref{t:Vuvol} then follows by
  summing this relation over $x\in V_\good$ and using \eqref{e:Vgoodbig}.
\end{proof}

We now prove the concentration of $n_i(\mathcal D^u)$ around its mean.
\begin{theorem}
  \label{t:niconc}
  Let $G$ be a $d$-regular {(multi)}graph on $n$ vertices satisfying
  $\lambda_G\ge\alpha >0$. Then, for every $\varepsilon \in (0,\tfrac
    14)$, and for every $i\in\{0,\dots,d\}$,
  \begin{equation}
    P^G\big[|n_i(\mathcal D^u)-E^G[n_i(\mathcal D^u)]|\ge
      n^{1/2+\varepsilon }\big]
    \le c_\varepsilon  e^{-cn^{\varepsilon }}.
  \end{equation}
\end{theorem}

To prove this lemma we use the following concentration theorem for (not
  necessarily independent) random variables that we learnt from
\cite{McD98}. Consider a sequence $W=(W_1,\dots,W_M)$ of random
variables, all taking values in some space $\mathcal A$. Let
$f:\mathcal A^M\to \mathbb R$ be a bounded function. For
$k\in \{1,\dots,M\}$ and $y_1,\dots,y_{k-1}\in \mathcal A^{k-1}$ we define
\begin{equation}
  \begin{split}
    r_k&(y_1,\dots,y_{k-1})
    \\&=\sup_{y,y'\in \mathcal A}
    \big|\mathbb E[f(W)|W_k=y,W_i=y_i\forall i<k]-
    \mathbb E[f(W)|W_k=y',W_i=y_i\forall i<k]\big|
  \end{split}
\end{equation}
and set
\begin{equation}
  R^2=\sup\Big\{ \sum_{k=1}^M r_k^2(y_1,\dots,y_{k-1}):y_1,\dots,y_{M-1}\in
    \mathcal A\Big\}.
\end{equation}

\begin{lemma}[Theorem 3.7 of \cite{McD98}]
  \label{l:concentration}
  Let $W=(W_1,\dots,W_M)$ be as above. Then
  \begin{equation}
    \mathbb P[|f(W)-\mathbb Ef(W)|\ge t]\le 2e^{-2t^2/R^2}.
  \end{equation}
\end{lemma}

\begin{proof}[Proof of Theorem~\ref{t:niconc}]
  To apply Lemma~\ref{l:concentration}, we need the following
  construction similar to Section~4 of \cite{CTW11}. Let
  $\ell=n^\varepsilon  $ for $\varepsilon $ from the statement of
  Theorem~\ref{t:niconc}. On an auxiliary probability space $(\Omega ,Q)$,
  define $(Z^i,i\in \mathbb N)$ to be a collection of i.i.d.~uniformly
  chosen vertices of $G$. Given the collection $(Z_i)$, let $(Y_i:i\ge 1)$
  be (conditionally) independent family of elements of $D([0,\ell],G)$
  such that $Y^i$ is distributed according to the random walk bridge
  $P^{G,\ell}_{Z^{i-1},Z^{i}}$ {(see Section~\ref{ss:rwnot} for the
      definition)}. We define $\mathcal X\in D([0,\infty),G)$ to be the
  concatenation of $Y^i$'s,
  \begin{equation}
    \mathcal X(t)=Y^i(t-(i-1)\ell), \qquad \text{when }
    (i-1)\ell \le t  < i\ell.
  \end{equation}
  We use $\mathcal P^G$ to denote the distribution of $\mathcal X$ on
  $D([0,\infty),G)$, $\mathcal P^G= Q\circ \mathcal X^{-1}$.
  $\mathcal P^{G,T}$ stands for its restriction to $D([0,T],G)$. The
  measure $\mathcal P^{G,un}$ approximates well $P^{G,un}$ if $\ell$ is
  large enough as follows from the next lemma whose proof is postponed to
  the end of this section.
  \begin{lemma}
    \label{l:equiv}
    Assume that  $\lambda_G>\alpha $ and $\ell = n^\varepsilon$. Then there exist constant
    $c_{\alpha ,\varepsilon  }$ and $c'_{\alpha ,\varepsilon  }$ such that for every
    $u> 0$ and all $n$ satisfying $ne^{-\ell \alpha }<1/2$,
    $\mathcal P^{G,un}$ and $P^{G,un}$ are equivalent and
    \begin{equation}
      \Big|\frac {\d \mathcal P^{G,nu}}{\d P^{G,nu}}-1\Big|
      \le c'_{\alpha,\epsilon} u e^{-c_{\alpha,\epsilon}\ell }.
    \end{equation}
  \end{lemma}

  To be able to apply Lemma~\ref{l:concentration}, more precisely to
  estimate the functions $r_k$, we do not want  $|\Ran Y^i|$ to be too
  large. Therefore, we define $\bar Y^i\subset V$ to be the set of first
  $2\ell$ vertices visited by $Y^i$,
  \begin{equation}
    \bar Y^i=\{Y^i_t,t\le (\tau_{2\ell}(Y^i) \wedge \ell)\},
  \end{equation}
  where $\tau_{k}(Y^i)$ denotes the time of $k$-th step $Y^i$ (defined to
    be infinite if $Y^i$ makes less than $k$ steps). Obviously
  $\bar Y^i \subset \Ran Y^i$. On the other hand, it can be proved as in
  Lemma~4.2 of \cite{CTW11}, that
  \begin{equation}
    \label{e:numjumps}
    Q(\Ran Y^i\neq \bar Y^i)
    \le \sup_{x,y\in V}P^{G,\ell}_{xy}[N_\ell\ge 2\ell]
    \le c e^{-c'\ell}.
  \end{equation}
  (Remark that $Y^i$ has the law of the random walk bridge, and thus
    \eqref{e:numjumps} is not just a direct consequence of large
    deviation estimate for a Poisson random variable.)

  We may now prove Theorem~\ref{t:niconc}. Set $m=\lfloor un/\ell\rfloor$ and
  $u'= m \ell /n$. Let $N$  be the number of steps of $X$ between $u'n$
  and $un$. Since $un-u'n\le \ell$, by properties of Poisson random
  variables, $P^G[N\ge 2 \ell]\le e^{-c \ell}$. Between, $u'n$ and $un$
  the walk visits at most $N$ sites, therefore
  \begin{equation}
    |n_i(\mathcal D^u)-n_i(\mathcal D^{u'})|\le (d+1) N,
  \end{equation}
  and
  \begin{equation}
    |E^G[n_i(\mathcal D^u)]-E^G[n_i(\mathcal D^{u'})]|\le (d+1) E^G[N]\le
    (d+1)\ell.
  \end{equation}
  Thus, for $n \geq c_{\alpha,\epsilon}$,
  \begin{equation}
\label{e:niEni}
    \begin{split}
      P^G&\big[|n_i(\mathcal D^u)-E^G[n_i(\mathcal D^u)]|\ge
        n^{1/2+\varepsilon }\big]
      \\&\le
      P^G\big[|n_i(\mathcal D^{u'})-E^G[n_i(\mathcal D^{u'})]|\ge
        n^{1/2+\varepsilon }-3(d+1)\ell \big] + P^G[N\ge 2\ell].
      \\&\le
      P^G\big[|n_i(\mathcal D^{u'})-E^G[n_i(\mathcal D^{u'})]|\ge
        \tfrac 12n^{1/2+\varepsilon } \big] + e^{-c \ell}.
    \end{split}
  \end{equation}
  Using Lemma~\ref{l:equiv}, denoting by $\mathcal E^G$ the expectation
  corresponding to $\mathcal P^G$,
  \begin{equation}
    \big|E^G[n_i(\mathcal D^{u'})]-\mathcal E^G[n_i(\mathcal D^{u'})]\big|\le
     c_{\alpha,\epsilon}nu' e^{-c'_{\alpha,\epsilon}\ell},
  \end{equation}
  and thus, for $n \geq c_{\alpha,\epsilon}$,
  \begin{equation}
\label{e:niprimeE}
    \begin{split}
      P^G\big[&|n_i(\mathcal D^{u'})-E^G[n_i(\mathcal D^{u'})]| \geq
        \tfrac 12n^{1/2+\varepsilon } \big]
      \\&\le
      \mathcal P^G\big[|n_i(\mathcal D^{u'})-\mathcal E^G[n_i(\mathcal D^{u'})]|\ge
        \tfrac 14n^{1/2+\varepsilon } \big]+ c_{\alpha,\epsilon} u' e^{-c'_{\alpha,\epsilon} \ell}.
    \end{split}
  \end{equation}

  Observe that under $\mathcal P^G$,
  $\mathcal V^{u'}=V\setminus \cup_{i\le m} \Ran Y^i$. Let
  $\bar{\mathcal  V}$ be the vacant set left by $\bar Y^i$'s,
  $\bar{\mathcal V}=V\setminus \cup_{i\le m} \bar Y^i$, and denote by
  $\bar{\mathcal D}$ the degree sequence of the graph with set of
  vertices $V$ and edge set
  $\{\{x,y\}\in \mathcal E:x,y\in \bar{\mathcal V}\}$,
  cf.~\eqref{e:vacantgraph}. By \eqref{e:numjumps}, we have then
  \begin{equation}
    Q[\mathcal D^{u'}\neq \bar {\mathcal D}]\le c m
    e^{-c'\ell}\le c e^{-c''\ell}.
  \end{equation}
  Therefore, for $n \geq c_{\alpha}$,
  \begin{equation}
    \big|\mathcal E^G[n_i(\mathcal D^{u'})]-Q[n_i(\bar{\mathcal
          D})]\big|
    \le n Q[\mathcal D^{u'}\neq \bar {\mathcal D}]
    \le c e^{-c'\ell},
  \end{equation}
  and thus
  \begin{equation}
\label{e:nipEcal}
    \begin{split}
      \mathcal P^G&\big[|n_i(\mathcal D^{u'})-\mathcal E^G[n_i(\mathcal D^{u'})]|\ge
        \tfrac 1 4n^{1/2+\varepsilon } \big]
      \\&\le
      Q\big[|n_i(\bar{\mathcal D})-Q[n_i(\bar{\mathcal D})]|\ge
        \tfrac 18 n^{1/2+\varepsilon } \big] + c e^{-c'\ell}.
    \end{split}
  \end{equation}

  We now apply Lemma~\ref{l:concentration} with $M=m$, $W_i=\bar Y^i$,
  $f=n_i(\bar {\mathcal D})$, and $\mathcal A$ being the set of subsets
  of $V$ with at most $2\ell$ elements.  Writing $\boldsymbol y_k=(y_1,\dots,y_k)$,
  $\boldsymbol y_k'=(y_1,\dots,y_{k-1},y'_k)$, and
  $\boldsymbol Y_k=(\bar Y_1,\dots,\bar Y_k)$, we claim that
  \begin{equation}
    \label{e:yuoi}
    r_k(\boldsymbol y_{k-1})=\sup_{y,y'\in \mathcal A}
    |Q[n_i(\bar {\mathcal D})|\boldsymbol Y_k=\boldsymbol y_k]
    -Q[n_i(\bar {\mathcal D})|\boldsymbol Y_k=\boldsymbol y'_k]|\le 2
    (d+1)\ell.
  \end{equation}
  Indeed, by conditioning also on the values of $\bar Y^{k+2},\dots,\bar Y^m$, we
  observe that the difference
  \begin{equation}
    |Q[n_i(\bar {\mathcal D})|\boldsymbol Y_k=\boldsymbol
      y_k,\bar Y^{k+2},\dots,\bar Y^m]
    -Q[n_i(\bar {\mathcal D})|\boldsymbol Y_k=\boldsymbol
      y'_k,\bar Y^{k+2},\dots,\bar Y^m]|
  \end{equation}
  cannot be larger than $(d+1)(|\bar Y^k|+|\bar Y^{k+1}|)\le 2(d+1)\ell$.
  The inequality \eqref{e:yuoi} then follows by integrating over
  $\bar Y^{k+2},\dots,\bar Y^m$.

  From \eqref{e:yuoi} it follows that we can apply
  Lemma~\ref{l:concentration} with
  $R^2=m (d+1)^2 \ell^2 = c n^{1+\varepsilon }$, yielding
  \begin{equation}
    Q\big[|n_i(\bar{\mathcal D})-Q[n_i(\bar{\mathcal D})]|\ge
      \tfrac 18n^{1/2+\varepsilon } \big]\le c e^{-c
      n^{1+2\varepsilon}/n^{1+\varepsilon  }} \le c e^{-c n^\varepsilon }.
  \end{equation}
  This, together with \eqref{e:niEni}, \eqref{e:niprimeE} and \eqref{e:nipEcal}
 completes the proof of Theorem~\ref{t:niconc}.
\end{proof}

\begin{proof}[Proof of Lemma~\ref{l:equiv}]
  Let $u'$ be the smallest number greater or equal to $u$, such that $u'n$
  is an integer multiple of $\ell$, and set $m=u'n/\ell$. Let further $A$
  be an arbitrary $\mathcal F_{un}$-measurable subset of $D([0,u'n],V)$.
  Since $P^{G,un}$ and $\mathcal P^{G,un}$ are the restrictions of the
  measures $P^{G,u'n}$ and $\mathcal P^{G,u'n}$   to $D([0,un],V)$, it is
  sufficient to prove the lemma with $u$ replaced by $u'$. To this end we
  write
  \begin{equation}
    \label{e:PP0}
    P^{G,u'n}[A]=\sum_{x_0,\dots,x_{m}\in V}
    P^{G,u'n}[A|X_{i\ell}=x_i,0\le i\le m]
    P^{G,u'n}[X_{i\ell}=x_i,0\le i\le m].
  \end{equation}
  By the Markov property
  \begin{equation}
    \label{e:PP}
    P^{G,u'n}[X_{i\ell}=x_i,0\le i\le m] =
    \pi(x_0)\prod_{k=0}^{m-1}
    P^\ell_{x_{k}} [X_\ell =x_{k+1}].
  \end{equation}
  The construction of the measure $\mathcal P^{G,u'n}$ implies that
  \begin{equation}
    \begin{split}
      \mathcal P^{G,u'n}[A|X_{i\ell}=x_i,0\le i\le m]&=
      P^{G,u'n}[A|X_{i\ell}=x_i,0\le i\le m],\\
      \label{e:QQ}
      \mathcal P^{G,u'n}[X_{i\ell}=x_i,0\le i\le 2m] &=
      \prod_{k=0}^{m} \pi (x_{k}).
    \end{split}
  \end{equation}
  Comparing \eqref{e:PP} and \eqref{e:QQ}, it remains to control the ratio
  $P^\ell_x[X_\ell=y]/\pi (y)$. However, by \eqref{e:I} and the assumption
  of the lemma,  $|P^{G,\ell}_x[X_\ell=y]/\pi (y)-1|\le n e^{-\alpha  \ell}$.
  This leads to
  \begin{equation}
    (1-ne^{\alpha \ell})^m\le \frac{\mathcal P^{G,u'n}[A]}{P^{G,u'n}[A]}\le (1+n e^{-\alpha \ell})^m
  \end{equation}
  Since $\ell = n^\varepsilon $ and $ne^{-\ell \alpha }\le \frac 12$ by
  the assumptions of the lemma, it immediately follows that $P^{G,u'n}$
  and $\mathcal P^{G,u'n}$ are equivalent. A change of constants
  accommodating the terms polynomial in $n$ then completes the proof.
\end{proof}

\section{Proofs of Theorems~\ref{t:criticalwindow} and \ref{t:outofwindow}}
\label{s:proofs}
We now have all tools that we need to show all the main results of this
paper. As a direct consequence of  Theorems~\ref{t:Vuvol},
\ref{t:niconc} and the fact \eqref{e:gap}, we get
that $\bbP_{n,d}$-a.a.s.
\begin{equation}
  \label{e:nconv}
  |n_i(\mathcal D^u)-n d_i^u|\le cn^{1/2}\log^5 n ,
  \qquad \text{for all $0\le i \le d$,}
\end{equation}
where $d_i^u$ is defined in \eqref{e:dui}.
Hence, $\bbP_{n,d}$-a.a.s,
$\lim_{n\to\infty }n^{-1}n_i(\mathcal D^u) = d_i^u$. The constant
$Q(\mathcal D^u)$ (see \eqref{e:Q}) can be written as
\begin{equation}
  Q(\mathcal D^u)=
  \frac{\sum_{x=1}^n \mathcal D^u(x)^2}
  {\sum_{x=1}^n \mathcal D^u(x)}-2
  =\frac{\sum_{i=1}^d i^2 n_i(\mathcal D^u)}
  {\sum_{i=1}^d i n_i(\mathcal D^u)} -2.
\end{equation}
Thus, for $n > c_u$, on the event in \eqref{e:nconv} we have
\begin{equation}
  \label{e:Qest}
  \Big| Q(\mathcal D^u) - \big(p_u(d-1) -1\big) \Big| \leq c_u n^{-1/2}\log^5 n.
\end{equation}

The value $u_\star$ given by \eqref{e:ustar} satisfies
$p_{u_\star}(d-1)-1=0$. Therefore, when $u_n\to u_\star$, we obtain by
expanding the exponential in the definition \eqref{e:pu} of $p_u$ around
$u_\star$,
\begin{equation}
  \Big|Q(\mathcal D^{u_n}) - (u_\star - u_n) \frac{(d-2)^2}{d(d-1)} \Big|
  \leq c \big((u_\star -u_n)^2 + n^{-1/2}\log^5 n\big).
\end{equation}

This implies that when $u_n$ is in the critical window of
Theorem~\ref{t:criticalwindow}, that is
$|n^{1/3}(u_\star - u_n)|\le \lambda $, then $Q(\mathcal D^{u_n})$ is in
the critical window of Theorem~\ref{t:rgds}, that is
$n^{1/3}|Q(\mathcal D^{u_n})|\le \lambda '$, $\bbP_{n,d}$-a.a.s.
Theorem~\ref{t:criticalwindow} then follows directly from
Theorem~\ref{t:rgds}(i) together with Proposition~\ref{p:CF} and the
remark following \eqref{e:simple}.

Very similar reasoning apply when proving Theorems~\ref{t:outofwindow}
and~\ref{t:supercritical}. We should only identify the constants of
Theorem~\ref{t:rgds}. Easy computations give
\begin{equation}
  \lambda = e^{-u\frac{d-2}{d-1}} d p_u, \qquad
  \beta  = e^{-u\frac{d-2}{d-1}} d(d-1)(d-2) p_u^3,
\end{equation}
and thus
\begin{equation}
  v_n=2n \lambda^2 \beta^{-1 }Q(\mathcal D^{u_n})=2n (u_\star - u_n) \frac{d-2}{(d-1)^2} e^{-u_\star
    \frac{d-2}{d-1}}(1+o(1)).
\end{equation}
Replacing $u_\star - u_n$ by $\omega_n n^{-1/3}$,
Theorem~\ref{t:outofwindow}(a) follows. It can also be seen that
$\sqrt {n/Q(\mathcal D^{u_n})}$ is of order $n^{2/3}\omega_n^{-1/2}$,
implying Theorem~\ref{t:outofwindow}(b).

Finally to identify $\rho $ of Theorem~\ref{t:supercritical}. We observe
that $g(x)$ of Theorem~\ref{t:rgds} is given by
\begin{equation}
  g(x)=\sum_{i=0}^d d^u_ix^i=e^{-u\frac{d-2}{d-1}}(x p_u + (1-p_u))^d.
\end{equation}
After few simplifications, $\xi $ of Theorem~\ref{t:rgds} is the unique
solution in $(0,1)$ of the equation
\begin{equation}
  (x p_u + (1-p_u))^{d-1} = x,
\end{equation}
and $\rho $ is given by
\begin{equation}
  \label{e:rhodef}
  \rho = 1-g(\xi ).
\end{equation}
This completes the proofs of all three main theorems.

\medskip

\begin{remark}
  (1) Theorem~\ref{t:Vuvol} raises the question of what is the right
  magnitude of deviations in $E^G[n_i(\mathcal{D}^u)]$ under the law
  $\bar{\mathbb{P}}_{n,d}$? If indeed it is of order $n^{1/2}$ (without
    power-log corrections), then it would be interesting to investigate
  whether this quantity satisfies a central limit theorem when properly
  rescaled.

  (2) As established in Proposition~\ref{p:CF} and Theorems~\ref{t:Vuvol}
  and \ref{t:niconc}, we can reduce the study of $\mathcal{V}^u$ to
  questions on the behaviour of random graphs with prescribed degree
  sequences. Although the results in \cite{MR95,JL09,HM10} provide very
  fine information about such graphs, there are several questions
  concerning them which are still open. For instance, one could give a
  better description of the geometry of their critical components, their
  diameters, spectral gaps, etc.

  (3) It is interesting to notice that the statements
  \eqref{e:supercritical} and \eqref{e:subcritical} were established in
  \cite{CTW11} for the vacant set left by random walk on other sequences
  of graphs, such as large girth expanders. Is it possible to extend the
  results of the current paper on the critical behaviour of $\mathcal{V}^u$
  to this more general setting?
\end{remark}

%%%% END OF 04-rw.tex %%%%>>>

%%%% START OF 09-appendix.tex %%%%<<<
\appendix
\section{Properties of the quasi-stationary distribution}

We establish here few results for arbitrary reversible irreducible
continuous-time Markov chains on a finite state space. These results are
natural but we have not found any suitable formulation in the literature.

Let $V$ be a finite set and let $\mathcal L$ be  the generator of a
reversible irreducible continuous-time Markov chain $X$ on $V$, and let
$\pi(x)$ be its invariant measure. We use $\<f,g\>$ to denote the usual
scalar product on $L^2(V,\pi )$, $\<f,g\>=\sum_{x\in V} f(x)g(x)\pi (x)$.
The operator $-\mathcal L$ is symmetric in $L^2(V,\pi )$ and has real
eigenvalues $0=\lambda_1<\lambda_2\le \dots \le \lambda_{|V|}$ and
corresponding orthonormal eigenvectors $v_1$, \dots,  $v_{|V|}$.

For $B\subset V$, we use $\mathcal L^B$ to denote the generator of the
Markov chain $X$ killed on hitting $B$, viewed as an operator on
$L^2(V\setminus B, \pi|_{V\setminus B} )$. Let
$0<\lambda^B_1<\lambda^B_2\le \dots\dots \lambda^B_{|V\setminus B|}$, and
$v_1^B$, \dots,  $v_{|V\setminus B|}^B$ denote the eigenvalues and
eigenvectors of $-\mathcal L^B$.

The \emph{quasi-stationary distribution} $\sigma_B$ is related to the
eigenvector of $-\mathcal L^B$ corresponding to $\lambda^B_1$ and is given
by
\begin{equation}
  \label{e:sigmavone}
  \sigma_B(y)=\frac{v^B_1(y)\pi (y)}{\<v_1^B,\boldsymbol 1\>}=
  \frac{\<v_1^B,\delta_y\>}{\<v_1^B,\boldsymbol 1\>},
\end{equation}
with $\boldsymbol 1$ denoting the constant one function. Inverting this
relation we get
\begin{equation}
  \label{e:vonesigma}
  v^B_1(x)=\frac{\sigma_B(x)}{\pi (x)} \Big(\sum_{x\in V\setminus B}
    \frac{\sigma_B (x)^2}{\pi (x)}\Big)^{-1/2}.
\end{equation}

\begin{lemma}
  \label{l:AB} For every $B\subset V$,
  \begin{equation}
    \label{e:AB2}
    \lambda^B_2 - \lambda^B_1 \geq \lambda_2 - \frac{1}{E[H_B]}.
  \end{equation}
\end{lemma}

\begin{proof}[Proof of Lemma~\ref{l:AB}.]
  Since $\mathcal L^B$ can be viewed as a sub-matrix of  $\mathcal L$, by
  the eigenvalue interlacing inequality (cf.~\cite{Hae95},
    Corollary~2.2), we have $\lambda^B_2 \geq \lambda_2$. On the other
  hand,  by \cite{AB93} Lemma~2 and the paragraph following equation~(12),
  \begin{align}
    \lambda^B_1 = \frac{1}{E_{\sigma_B}[H_B]} \leq \frac{1}{E[H_B]}.
  \end{align}
  Combining these two inequalities we obtain Lemma~\ref{l:AB}.
\end{proof}

\begin{lemma}
  \label{l:quasigen}
  Suppose that for $t>0$ and $\varepsilon \in (0,1/2)$
  \begin{equation}
    \label{e:appass}
    e^{-t(\lambda_2^B-\lambda_1^B)} |V\setminus B|
    \Big(\sup_{x\in V\setminus B}
      \frac {\sigma_B (x)}{\sqrt{\pi (x)}}\Big)^2
      \le \varepsilon
      \inf_{x\in V\setminus B} \frac {\sigma_B (x)}{\sqrt{\pi (x)}}.
  \end{equation}
  Then,
  \begin{equation}
    \label{e:quasigen}
    \sup_{x, y \in {V} \setminus B} \big| P_x[X_{t} = y| H_B > {t}] -
    \sigma_B(y) \big|
    \leq 4 \varepsilon.
  \end{equation}
\end{lemma}

\begin{proof}
  In the proof we will only use the eigenvalues and eigenvectors of
  $-\mathcal L^B$, therefore we omit the superscript $B$ from the notation.
  Similarly, we write $\sigma $ for $\sigma_B$ and define
  $m=|V\setminus B|$.
  By the usual spectral decomposition formula,
  \begin{equation}
    \begin{split}
      P_x[X_t=y, H_B>t] &= {(e^{t\mathcal L^B} \delta_y)(x)}=
      \sum_{k=1}^m e^{-\lambda_k t} v_k(x)\<v_k,\delta_y\>,
      \\P_x[H_B>t] &= (e^{t\mathcal L^B} \boldsymbol 1)(x)=
      \sum_{k=1}^m e^{-\lambda_k t} v_k(x)\<v_k,\boldsymbol 1\>,
    \end{split}
  \end{equation}
  where $\delta_y$ is the indicator function of $y$. For $f\in L^2(\pi )$,
  define $\psi_f=\sum_{k=2}^m e^{-(\lambda_k-\lambda_1)t} \<v_k,f\> v_k$.
  Then $e^{t\mathcal L^B}f=e^{-\lambda_1 t}(v_1 \<v_1,f\> +\psi_f)$, and
  by Pythagoras' theorem
  \begin{equation}
    \label{e:normpsi}
    \|\psi_f\|_{L^2(\pi )} \le e^{-(\lambda_2-\lambda_1)t}\|f\|_{L^2(\pi)}.
  \end{equation}
  Using this notation and the definition of the conditional probability,
  \begin{equation}
    P_x[X_t=y|H_B>t]=
    \frac{v_1(x)\<v_1,\delta_y\> + \psi_{\delta_y}(x)}
    {v_1(x)\<v_1,\boldsymbol 1\> + \psi_{\boldsymbol 1}(x)}.
  \end{equation}
  Applying \eqref{e:sigmavone} we get after an easy algebra
  \begin{equation}
    \label{e:apa}
    P_x[X_t=y|H_B>t]-\sigma(y)=
    \frac{\psi_{\delta_y}(x)-
      \frac{\<v_1,\delta_y\>}{\<v_1,\boldsymbol 1\>}
      \psi_{\boldsymbol 1}(x)}
    {v_1(x)\<v_1,\boldsymbol 1\>
      (1+\frac{\psi_{\boldsymbol 1}(x)}
        {v_1(x)\<v_1,\boldsymbol 1\>})}.
  \end{equation}

  Let $f$ stand either for $\delta_y$ or $\boldsymbol 1$. Then
  $\|f\|_{L^2(\pi )}\le 1$, which directly implies
  $|\psi_f(x)|\le \pi (x)^{-1/2}e^{-(\lambda_2-\lambda_1)t}$.
  From \eqref{e:normpsi}, \eqref{e:vonesigma}, using the assumption
  \eqref{e:appass}, we obtain
  \begin{equation}
    \frac{\psi_{f}(x)} {v_1(x)\<v_1,\boldsymbol 1\>}
    \le \frac{e^{-(\lambda_2-\lambda_1)t}
      \sum_z\frac{\sigma(z)^2}{\pi (z)}}
    {\frac{\sigma (x)}{\sqrt{\pi (x)}}}
    \le \frac{e^{-(\lambda_2-\lambda_1)t}
      m \sup_z\big(\frac{\sigma(z)}{\sqrt {\pi (z)}}\big)^2}
    {\inf_z\frac{\sigma (z)}{\sqrt{\pi (z)}}} \le \varepsilon.
  \end{equation}
  Using $\varepsilon <1/2$, this implies that the absolute value of \eqref{e:apa}
  can be bounded from above by
  \begin{equation}
    \frac {2(|\psi_{\delta_y}(x)|+|\psi_{\boldsymbol
          1}(x)|)}{v_1(x)\<v_1,\boldsymbol 1\>}\le 4\varepsilon.
  \end{equation}
  This completes the proof of Lemma~\ref{l:quasigen}.
\end{proof}

%%%% END OF 09-appendix.tex %%%%>>>

\bibliographystyle{jcamsalpha}
\bibliography{regular}

\end{document}